\documentclass[11pt]{amsart}
\usepackage{amssymb,amscd}
\usepackage{hyperref}
\usepackage{xcolor}
\usepackage{bbm,color}
\usepackage{todonotes}
\usepackage[right=3cm, left=3cm]{geometry}

\hypersetup{
    colorlinks,
    linkcolor={red!50!black},
    citecolor={blue!50!black},
    urlcolor={blue!80!black}
}

\usepackage[all,cmtip]{xy}
\usepackage{textcomp}

\usepackage{tikz}
\usepackage{color}
\usetikzlibrary{shapes,arrows}
\usepackage{cleveref}
\usepackage{float}

\numberwithin{equation}{section}
\newtheorem{thm}[equation]{Theorem}
\newtheorem{lem}[equation]{Lemma}
\newtheorem{prop}[equation]{Proposition}
\newtheorem{cor}[equation]{Corollary}
\newtheorem{bad-conj}[equation]{False Conjecture}

\theoremstyle{definition}
\newtheorem{defn}[equation]{Definition}

\theoremstyle{remark}
\newtheorem{ex}[equation]{Example}
\newtheorem{rem}[equation]{Remark}

\newcommand{\m}{\mathfrak{m}}
\newcommand{\rto}{\to}
\renewcommand{\L}{\mathcal{L}}
\newcommand{\LL}{\mathbb{L}}
\newcommand{\RR}{\mathbb{R}}

\newcommand{\A}{\mathcal{A}}
\DeclareMathOperator{\Mod}{Mod}
\newcommand{\p}{\mathfrak{p}}
\renewcommand{\O}{\mathcal{O}}
\newcommand{\bs}{\setminus}
\newcommand{\Set}{\mathbf{Set}}

\newcommand{\F}{\mathbb{F}}

\newcommand{\xra}{\xrightarrow}

\def\Tor{\mathrm{Tor}}
\def\Ext{\mathrm{Ext}}
\def\Hom{\mathrm{Hom}}

\def\Z{\mathbb{Z}}
\def\Q{\mathbb{Q}}
\def\F{\mathbb{F}}

\def\leq{\leqslant}
\def\geq{\geqslant}

\def\call{\mathcal{L}}

\def\ra{\rightarrow}

\def\HH{\mathsf{HH}}
\def\THH{\mathsf{THH}}

\def\Sh{\mathsf{Sh}}
\def\Shukla{\mathsf{Sh}}
\def\ie{\emph{i.e.}}
\def\eg{\emph{e.g.}}
\def\id{\mathrm{id}}

\def\ba#1\ea{\begin{align*}#1\end{align*}}
\def\[#1\]{\begin{align*}#1\end{align*}}

\begin{document}
\title[Splittings and calculational techniques for higher
$\THH$]{Splittings and calculational techniques for higher $\THH$} 
\author[Bobkova]{Irina Bobkova}
\address{School of Mathematics, Institute for Advanced Study,
  Princeton, NJ 08540, USA} 
\email{ibobkova@ias.edu}
\author[H\"oning]{Eva H\"oning}
\address{Max-Planck-Institut f\"ur Mathematik, Vivatsgasse 7, 53111
  Bonn, Germany}
\email{hoening@mpim-bonn.mpg.de}
\author[Lindenstrauss]{Ayelet Lindenstrauss}
\address{Mathematics Department, Indiana University, 831 East Third Street,
  Bloomington, IN 47405, USA}
\email{alindens@indiana.edu}
\author[Poirier]{Kate Poirier}
\address{Mathematics Department, New York City
College of Technology, CUNY, 300 Jay Street, Brooklyn, NY 11201, USA}
\email{kpoirier@citytech.cuny.edu}
\author[Richter]{Birgit Richter}
\address{Fachbereich Mathematik der Universit\"at Hamburg,
Bundesstra{\ss}e 55, 20146 Hamburg, Germany}
\email{birgit.richter@uni-hamburg.de}
\author[Zakharevich]{Inna Zakharevich}
\address{Department of Mathematics, Cornell University, 
Malott Hall 587, Ithaca, NY 14853, USA}
\email{zakh@math.cornell.edu}
\date{\today}

\begin{abstract}
Tensoring finite pointed simplicial sets $X$ with commutative ring
spectra $R$ yields important homology theories such as (higher)
topological Hochschild homology and torus homology. We prove several
structural properties of these constructions relating $X \otimes (-)$
to $\Sigma X \otimes (-)$ and we establish splitting results. This allows
us, among other important examples, to determine $\THH^{[n]}_*(\Z/p^m;
\Z/p)$ for all $n \geq 1$ and for all $m \geq 2$. 
\end{abstract}
\thanks{}
\maketitle

\section*{Introduction}
For any (finite)
pointed simplicial set $X$ one can define the tensor product of $X$
with a commutative ring spectrum $A$, $X \otimes A$, where the case
$X=S^1$ gives topological Hochschild homology of $A$. More generally,
for any sequence 
of maps of commutative ring spectra $R \ra A \ra C$ we can define
$\L^R_X(A; C)$, the Loday construction with respect to $X$ of
$A$ over $R$ with coefficients in $C$. Important examples are $X =
S^n$ or $X$ a torus. The construction specializes to $X \otimes A$ in the 
case $\L_X^S(A; A)$. For details see Definition \ref{def:loday}. 

An important question about the Loday construction concerns the
dependence on $X$: 
Given $X, Y\in s\Set_*$, with $\Sigma X \simeq \Sigma Y$, does that imply that
$\L^R_X(A) \simeq \L^R_Y(A)$?  If it does,  the Loday construction would be a
``stable invariant''. Positive cases arise from the work of Berest,
Ramadoss and Yeung \cite[Theorem 5.2]{BRY}: They identify the homotopy groups of
the Loday construction over an $X \in s\Set_*$ of a Hopf algebra over
a field with representation homology of the Hopf algebra with respect
to $\Sigma(X_+)$. In \cite{DT} Dundas and Tenti prove that stable
invariance holds if $A$ is a smooth algebra over a commutative ring $k$. However in
\cite{DT} they also provide a counterexample: $\L^{H\Q}_{S^2 \vee S^1 \vee
  S^1}(H\Q[t]/t^2)$ is \emph{not} equivalent to $\L^{H\Q}_{S^1 \times
S^1}(H\Q[t]/t^2)$ even though $\Sigma(S^2 \vee S^1 \vee
  S^1) \simeq \Sigma(S^1 \times
S^1)$. Our juggling formula (Theorem \ref{thm:juggling}) and our
generalized Brun splitting (Theorem \ref{thm:brun}) relate the Loday
construction on $\Sigma X$ to that of $X$. One application among  
others of these results is to establish stable invariance in certain examples. 

For commutative $\F_p$-algebras $A$ one often observes a splitting of
$\THH(A)$ as $\THH(\F_p) \wedge_{H\F_p} \THH^{H\F_p}(HA)$, so
$\THH(A)$ splits 
as topological Hochschild homology of $\F_p$ tensored with the
Hochschild homology 
of $A$ \cite{ll2}. It is natural to ask in which generality such
splittings occur. If one replaces $\F_p$ by
$\Z$, then there are many counterexamples. For instance if
$A=\mathcal{O}_K$ is a number ring then $\THH_*(\mathcal{O}_K)$ is known
by \cite[Theorem 1.1]{LM} and is far from being equivalent to $\pi_*(\THH(\Z)
\wedge^L_{H\Z} \THH^{H\Z}(H\mathcal{O}_K))$ in general (see
\cite[Remark 4.12.]{hhlrz} for a concrete example). We prove several splitting
results for higher $\THH$ and use one of them to determine higher
$\THH$ of $\Z/p^m$ with $\Z/p$-coefficients for all $m \geq 2$.

\subsection*{Content}

We start with a brief recollection of the Loday construction in
Section \ref{sec:loday}. 

In \cite{blprz} we determined the higher Hochschild homology of
$R=\F_p[x]$ and of 
$R = \F_p[x]/x^{p^\ell}$, both
Hopf algebras. In Section \ref{sec:truncated} we generalize our results
to the cases $R = \F_p[x]/x^{m}$ if $p$ divides $m$, which is not a
Hopf algebra unless $m=p^\ell$ for some $\ell$.

Building on work in \cite{hhlrz} we prove a juggling formula (see
Theorem \ref{thm:juggling}): For every sequence of cofibrations of
commutative ring spectra  
$S  \ra R \ra A \ra B \ra C$ there is an equivalence
\[ \L^R_{\Sigma X}(B; C) \simeq \L^R_{\Sigma X}(A;
C) \wedge^L_{\L_X^A(C)} \L_X^B(C).\]
We explain in \ref{subsec:phony} what happens if one oversimplifies
this formula.

Using a geometric argument, Brun \cite{brun} constructs a spectral
sequence for calculating $\THH_*$-groups. We prove a 
generalization of his splitting and show that for any sequence of
cofibrations of commutative ring spectra 
$S \ra R \ra A \ra B$ we obtain a generalized spectrum-level Brun splitting 
(see Theorem \ref{thm:brun}) 
\ie, an equivalence of commutative $B$-algebra spectra 
\[ \L^R_{\Sigma X}(A; B) \simeq B \wedge^L_{\L^R_X(B)}  \L_X^A(B). \]
Note that $B$, which only appears at the basepoint on the left, now appears
almost everywhere on the right.  This splitting also gives rise to
associated spectral sequences for calculating higher
$\THH_*$-groups. 

We apply our results to prove a generalization of Greenlees' splitting formula
\cite[Remark 7.2]{greenlees}: 
For an augmented commutative $k$-algebra $A$ we obtain in Corollary
\ref{cor:augmented} that  
$$ \L_{\Sigma X}(HA; Hk) \simeq \L_{\Sigma X}(Hk)
\wedge^L_{Hk} \L_X^{HA}(Hk)$$ 
and if $A$ is flat as a $k$-module then this can also be written as
$$ \L_{\Sigma X}(HA; Hk) \simeq \L_{\Sigma X}(Hk)
\wedge^L_{Hk} \L^{Hk}_{\Sigma X}(HA; Hk),$$ 
where all the Loday constructions are over the same simplicial set.
For $X=S^n$, for example, and $A$ a flat augmented commutative
$k$-algebra this yields  Theorem \ref{prop:genGreenlees},
$$ \THH^{[n]}(A; k) \simeq \THH^{[n]}(k) \wedge^L_{Hk}
\THH^{[n], k}(A; k)$$
where  $\THH^{[n]}=\L_{S^n}$. 

Shukla homology is a derived version of Hochschild homology. We define
higher order Shukla homology in Section \ref{sec:shukla} and calculate
some examples that will be used in subsequent results. We prove that
the Shukla homology of order $n$ of a ground ring $k$ over a flat
augmented $k$-algebra is isomorphic to the 
reduced Hochschild homology of
order $n+1$ of the flat  augmented algebra (Proposition \ref{prop:hhsh}). 

Tate shows \cite{Tate} how to control Tor-groups for certain quotients
of regular local rings. We use this to develop a splitting on the
level of homotopy groups for $\THH(R/(a_1, \ldots, a_r); R/\mathfrak{m})$ if $R$
is regular local, $\mathfrak{m}$ is the maximal ideal and the $a_i$'s are in
$\mathfrak{m}^2$. 

We prove a splitting result for $\THH^{[n]}(R/a, R/p)$ in Section
\ref{sec:splitting}, where $R$ is a commutative ring and $p,a\in R$
are not zero divisors, 
$(p)$ is a maximal ideal, and $a \in
(p)^2$.  In this situation,
$$ \THH^{[n]}(R/a, R/p) \simeq \THH^{[n]}(R, R/p)
\wedge^L_{HR/p} \THH^{[n];R}(R/a, R/p).$$ 
In many cases the homotopy groups of the factors on the right hand
side can be completely determined. Among other important examples we
get explicit 
formulas for $ \THH^{[n]}(\Z/p^m, \Z/p)$ for 
all $n \geq 1$ and all $m \geq 2$ (compare Theorem \ref{thm:thhnzpmzp}):  
$$ \THH^{[n]}_*(\Z/p^m, \Z/p) \cong
\THH^{[n]}_*(\Z, \Z/p) \otimes_{\Z/p}
\THH_*^{[n],\Z}(\Z/p^m, \Z/p).$$ 
We know $\THH^{[n]}_*(\Z, \Z/p)$ from
\cite{dlr} and we determine $\THH_*^{[n],\Z}(\Z/p^m, \Z/p)$
explicitly for all $n$. 

This generalizes previous results
by Pirashvili \cite{pirashvili}, Brun \cite{brun}, and Angeltveit
\cite{angeltveit}   from $n = 1$ to all $n$. 

In Section \ref{sec:sums} we provide a splitting
result for commutative ring spectra of the form $A \times B$: we show in
Proposition 
\ref{prop:product}  that for any finite 
connected simplicial set $X$, we have 
  $$ \xymatrix@1{ {\L_X(A \times B)}
  \ar[r]^(0.42){\simeq} &
   {\L_X(A)\times \L_X(B). }  } $$ 

We present some sample applications of our splitting results in Section
\ref{sec:appl}:  a splitting of higher $\THH$ of ramified number rings 
with reduced coefficients (\ref{subsec:numberrings}), a version of
Galois descent for higher $\THH$ (\ref{subsec:gd}) and a calculation
of higher $\THH$ of function fields over $\F_p$ (\ref{subsec:ff}). We
close with a discussion of the case of 
higher $\THH$ of $\Z/p^m$ (with unreduced coefficients)
(\ref{subsec:thhnzpm}). 

\subsection*{Acknowledgements}
This paper grew out of a follow-up project to our \emph{Women in
  Topology I} project 
on higher order topological Hochschild homology \cite{blprz}. 
We were supported by an AIM SQuaRE grant that allowed us to meet in 2015, 2016
and 2018. AL acknowledges support by Simons Foundation grant 
359565.  BR thanks the Department of Mathematics of the Indiana
University Bloomington for invitations in 2016, 2017 and 2018. AL and
BR thank Michael Larsen for lessening their ignorance of algebraic number
theory. We thank Bj{\o}rn Dundas, Mike Mandell, and Brooke Shipley for helpful comments.

\section{The Loday construction: basic features} 
\label{sec:loday}
We recall some definitions
concerning the Loday construction and we fix notation. 

For our work we can use any good symmetric monoidal category of
spectra whose category of commutative monoids is Quillen equivalent to
the category of $E_\infty$-ring spectra, such as symmetric
spectra \cite{HSS}, orthogonal spectra \cite{mm} or $S$-modules
\cite{ekmm}. As parts of the paper require to work with a specific
model category we chose to work with the category of $S$-modules. 

Let $X$ be a finite pointed simplicial set and let $R \ra A \ra C$ be a sequence
of maps of commutative ring spectra. 

\begin{defn} \label{def:loday}
The Loday construction with respect to $X$ of
$A$ over $R$ with coefficients in $C$ is the simplicial commutative
augmented $C$-algebra spectrum $\L^R_X(A; C)$ whose $p$-simplices are 
$$ C \wedge \bigwedge_{x \in X_p \setminus *} A$$
where the smash products are taken over $R$.  
Here, $*$ denotes the basepoint of $X$ and we place a copy of $C$ at
the basepoint. As the smash product over $R$ is the coproduct in the
category of commutative $R$-algebra spectra, the simplicial structure
is straightforward: Face maps $d_i$ on $X$ induce multiplication in
$A$ or the $A$-action on $C$ if the basepoint is 
involved. Degeneracies $s_i$ on $X$ correspond to the insertion of
the unit maps $\eta_A \colon R \ra A$ over all $n$-simplices which
are not hit by $s_i\colon X_{n-1} \to X_n$.
\end{defn}

As defined above, $\L^R_X(A; C)$ is a simplicial commutative augmented
$C$-algebra spectrum. If $M$ is an $A$-module spectrum, then $\L^R_X(A; M)$ is
defined. By slight abuse of notation we won't distinguish $\L^R_X(A;
C)$ or $\L^R_X(A; M)$ from their geometric realization.

If
$X$ is an arbitrary pointed simplicial set, then we can write it as 
the colimit of its finite pointed subcomplexes and the Loday
construction with respect to $X$ can then also be expressed as the
colimit of the Loday construction for the finite subcomplexes. 

An important case is $X = S^n$. In this case we write $\THH^{[n],R}(A;
C)$ for $\L_{S^n}^R(A; C)$; this is the \emph{higher order topological
Hochschild homology of order $n$ of $A$ over $R$ with coefficients in
$C$}.

Let $k$ be a commutative ring, $A$ be a commutative $k$-algebra,
and $M$ be an $A$-module.  Then we define 
\[
  \THH^{[n],k}(A;M) := \L^{Hk}_{S^n}(HA;HM).
\]
If $A$ is flat over $k$, then $\pi_*\THH^{k}(A;M) \cong
\HH_*^{k}(HA;HM)$ \cite[Theorem IX.1.7]{ekmm} and this also holds for
higher order Hochschild homology in the sense of Pirashvili
\cite{pira-hodge}: $\pi_*\THH^{[n],k}(A;M) \cong  HH_*^{[n],k}(A;M)$
if $A$ is $k$-flat \cite[Proposition 7.2]{blprz}. 

To avoid visual clutter, given a commutative ring $A$ and an element
$a\in A$, we write $A/a$ instead of $A/(a)$.

\section{Higher $\THH$ of truncated polynomial
  algebras} \label{sec:truncated} 
When the Loday construction is viewed as a functor on pointed simplicial sets,
it transforms homotopy pushouts of pointed simplicial sets into
homotopy pushouts of Loday constructions. In \cite[Section 3]{veen}
Veen uses this to express higher $\THH$ as a 
``topological $\Tor$'' of a lower $\THH$, that is: for any commutative
$S$-algebra $A$,
$$\THH^{[n]}(A) \simeq A \wedge^L_{\THH^{[n-1]}(A)} A.$$
This yields a spectral sequence 
$$E^2_{s, *}=\Tor_s^{\THH_*^{[n-1]}(A)}(A_*, A_*) \Rightarrow \THH_*^{[n]}(A).$$
In particular cases, this spectral sequence collapses for all $n\geq 
1$ making it possible to calculate $\THH^{[n]}_*(A)$ as iterated $\Tor$'s
of $A_*$. 
 In
\cite[Figures 1 and 2]{blprz} we had a flow chart showing the results
of iterated $\Tor$ of $\F_p$ over  
some $\F_p$-algebras with a particularly convenient form.  We can do
similar calculations over any field: 

\begin{prop} \label{prop:flowcharts}
If $F$ is a field of characteristic $p$ and $|\omega|$ is even, there
  is a flow chart as in Figure 
  \ref{fig:flowchart-one} showing the calculation of iterated $\Tor$'s
  of $F$:  If $\mathcal{A}$ is a term in the $n$th generation in the flow
  chart, then  $\Tor^{\mathcal{A}}(F,F)$ 
is the tensor product of all the terms in the $(n+1)$st generation
that arrows from $\mathcal{A}$ point to.  Here $|\rho^0y| =|y|+1$, $|\epsilon 
  z|=|z|+1$ and $|\varphi^0 z|=2+p|z|$.

\begin{figure}[H]
    \begin{center}
\begin{tikzpicture}
\node at (0,0) (1) {$F[\omega]$};
\node at (1.5,0) (2) {$\Lambda(\epsilon\omega)$};
\draw [->](1) to (2);

\node at (3.5, 0) (3) {$\Gamma(\rho^0\epsilon\omega)\cong$};
\node at (6.1,-0.15) (35) {$\bigotimes\limits_{k\geq 0} F[\rho^k\epsilon\omega]/{(\rho^k\epsilon\omega)^p}$};
\draw [->](2) to  (3);

\node at (10.5, 1) (4) {$\bigotimes\limits_{k\geq 0}\Lambda(\epsilon\rho^k\epsilon\omega)$};
\draw [->] (35) to (4);

\node at (6.5, -1.7) (5) {$\bigotimes\limits_{k\geq 0}\Gamma(\varphi^0 \rho^k\epsilon\omega)\cong$};
\node at (10, -1.7) (6) {$\bigotimes\limits_{k,i\geq 0}F[\varphi^i\rho^k\epsilon\omega]/{(\varphi^i\rho^k\epsilon\omega)^p}$};
\draw [->] (35) to (6);

\node at (13.7, 1) (7) {$\cdots$};
\draw [->] (4) to (7);
\node at (13.7, -2) (8) {$\cdots$};
\node at (13.7, -1.4) (9) {$\cdots$};
\draw [->] (6) to (8);
\draw [->] (6) to (9);
\end{tikzpicture}   
\end{center}
\caption{Iterated $\Tor$ flow chart
}\label{fig:flowchart-one}
\end{figure}
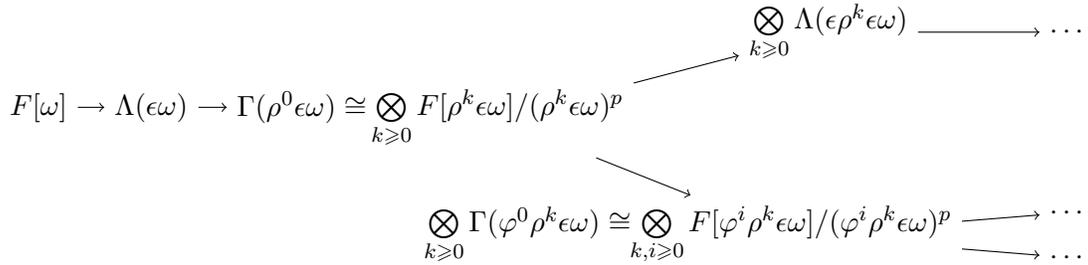

If $F$ a field of characteristic zero, the analogous flow chart for
$|x|$ even is 
  \begin{center}
    \begin{tikzpicture}
       \node at (0,0) (1) {$F[x]$};
       \node at (2,0) (2) {$\Lambda(\epsilon x)$};
       \node at (4,0) (3) {$F[\rho^0\epsilon x]$};
       \node at (6,0) (4) {$\Lambda(\epsilon \rho^0 \epsilon x)$};
       \node at (8,0) (5) {$\cdots$};
       \draw [->] (1) to (2);          
       \draw [->] (2) to (3);
       \draw [->] (3) to (4);
       \draw [->] (4) to (5);
    \end{tikzpicture} 
  \end{center}
\end{prop}

\begin{proof}
  In characteristic $p$, all divided power algebras split as tensor products of
  truncated polynomial algebras.  This allows us to use the resolutions of
  \cite[Section 
  2]{blprz}, where the tensor products in the respective bar constructions are all taken to be over $F$, to pass from each stage to the next.  
  
  In characteristic
  $0$, the $\Tor$ dual of an exterior algebra is a divided power algebra, but
  this is isomorphic to a polynomial algebra.  Thus the resolutions of \cite[Section
  2]{blprz}, with the tensor products in the bar constructions again taken to be over $F$,  can be used analogously to get the alternation between exterior and polynomial algebras.
\end{proof}

Let $x$ be a generator of even non-negative degree. In \cite[Theorem
8.8]{blprz} we calculated higher $\HH$ of truncated polynomial rings of the form
$\F_p[x]/x^{p^\ell}$ for any prime $p$. The decomposition due to
B\"okstedt which is described before the statement of the theorem
there does not work for $\F_p[x]/x^m$ when $m$ is not a power of $p$,
but we can nevertheless use a similar kind of argument to determine
higher $\HH$ of $\F_p[x]/x^m$ as long as $p$ divides $m$.    This
generalization of \cite[Theorem 8.8]{blprz} is  interesting because if
$m$ is not a power of $p$,  $\F_p[x]/x^m$  is no longer a Hopf
algebra, which the cases we discussed in \cite{blprz} were.   

In the following $\HH^{[n]}_*$ will denote Hochschild
homology groups of order $n$ whereas $\HH^{[n]}$ denotes the corresponding
simplicial object whose homotopy groups are $\HH^{[n]}_*$. 

\begin{thm}
Let $x$ be of even degree and let $m$ be a positive integer divisible by $p$.
 Then for all $n \geq 1$ 
$$ \HH^{[n],\F_p}_*(\F_p[x]/x^m) \cong \F_p[x]/x^m \otimes B''_n(\F_p[x]/x^m),$$
where 
$B''_1(\F_p[x]/x^m)\cong  \Lambda_{\F_p}(\varepsilon
  x) \otimes
\Gamma_{\F_p}(\varphi^0x)$,  with $|\varepsilon x| = |x| + 1$, $|\varphi^0x| = 2
+m|x|$ and 
$$B''_n(\F_p[x]/x^m)\cong \Tor_{*,*}^{B''_{n-1}(\F_p[x]/x^m)}(\F_p, \F_p).$$
\end{thm}

\smallskip 
Since $ \F_p[x]/x^m$ is monoidal over $\F_p$, this gives a higher
$\THH$ calculation, 
\begin{align} \label{eq:splittr}
\begin{split}
\THH^{[n]}_*(\F_p[x]/x^m)  &\cong \THH^{[n]}_*(\F_p) \otimes
\HH^{[n],\F_p}_*(\F_p[x]/x^m) \\
& \cong B^n_{\F_p}(\mu) \otimes  \F_p[x]/x^m \otimes
B''_n(\F_p[x]/x^m)
\end{split}
\end{align}
where $B^1_{\F_p}(\mu) \cong \F_p[\mu]$ with $|\mu|=2$ and
$B^n_{\F_p}(\mu) = \Tor_{*,*}^{B^{n-1}_{\F_p}(\mu)}(\F_p, \F_p)$ for
$n > 1$ (see \cite[Remark 3.6]{dlr}).

\begin{proof}
We use the standard resolution \cite[(1.6.1)]{ll} and get that the Hochschild
homology of $\F_p[x]/x^m$ is the homology of the complex
$$ \xymatrix@1{
{\ldots} \ar[r]^(0.3){0} & {\Sigma^{m|x|}\F_p[x]/x^m}
\ar[rr]^{\Delta(x,x)} & &
{\Sigma^{|x|} \F_p[x]/x^m} \ar[r]^{0} & {\F_p[x]/x^m.}}$$
Since $p$ divides $m$, we have $\Delta(x,x) = mx^{m-1} \equiv 0$ so the above
differentials are all trivial and
$$\HH^{\F_p}_*(\F_p[x]/x^m) \cong \F_p[x]/x^m \otimes
\Lambda_{\F_p}(\varepsilon x) \otimes \Gamma_{\F_p}(\varphi^0x)$$
at least as an $\F_p[x]/x^m$-module, with $|\varepsilon x| = |x| + 1$,
$|\varphi^0x| = 2 +m|x|$.  The map $G.$ from \cite[equation
(1.8.6)]{ll} embeds this small complex quasi-isomorphically with its
stated multiplicative structure into the standard Hochschild complex
for $\F_p[x]/x^m$.  It sends  $\F_p[x]/x^m$ to itself in degree zero
of the Hochschild complex and 
$$\varepsilon x\mapsto 1\otimes x-x\otimes 1,\quad \varphi^0x \mapsto
\sum_{i_0=1}^m\sum_{j_0=0}^1 (-1)^{1+j_0} x^{i_0-j_0}\otimes x^{m-i_0}
\otimes x^{j_0}$$ 
which generate an exterior and divided power subalgebra, respectively,
inside the standard Hochschild complex equipped with the shuffle
product.  The map $G.$ is  shown in \cite{ll} to be half of a chain
homotopy equivalence between the small complex and the standard
Hochschild complex.  So we get that  $\HH^{\F_p}_*(\F_p[x]/x^m)
=\HH^{[1],\F_p}_*(\F_p[x]/x^m) $ has the desired form as an algebra and
sits as a deformation retract inside the standard complex calculating it. 

For the higher $\HH_*$-computation we use that the $E^2$-term of the
spectral sequence for $\HH^{[2]}_*$
is $$ E^2_{*,*} = \Tor^{\F_p[x]/x^m \otimes
\Lambda_{\F_p}(\varepsilon x) \otimes
\Gamma_{\F_p}(\varphi^0x)}(\F_p[x]/x^m, \F_p[x]/x^m)$$
and as the generators $\varepsilon x$ and $\varphi^0x$ come from
homological degree one and two, the module structure of $\F_p[x]/x^m$
over $\Lambda_{\F_p}(\varepsilon x)$ and $\Gamma_{\F_p}(\varphi^0x)$
factors over the augmentation to $\F_p$. Therefore the above Tor-term splits as
\begin{equation}\label{eqn:form}
 \F_p[x]/x^m \otimes \Tor_{*,*}^{\Lambda_{\F_p}(\varepsilon
  x)}(\F_p, \F_p) \otimes
\Tor_{*,*}^{\Gamma_{\F_p}(\varphi^0x)}(\F_p, \F_p). 
\end{equation}

Now we can argue as in \cite{blprz} to show that there cannot be any
differentials or 
extensions in this spectral sequence: although we are calculating the
homology of the total complex of the bisimplicial $\F_p$-vector space
of the bar construction $B(\F_p[x]/x^m, \HH^{\F_p}(\F_p[x]/x^m),
\F_p[x]/x^m)$ which involves both vertical and horizontal boundary
maps, we can map the bar construction  
$$B(\F_p[x]/x^m, \F_p[x]/x^m \otimes
\Lambda_{\F_p}(\varepsilon x) \otimes \Gamma_{\F_p}(\varphi^0x), \F_p[x]/x^m)$$
quasi-isomorphically into it, and the latter complex involves only
non-trivial horizontal maps (all vertical differentials vanish) and
has homology exactly equal to the algebra in equation
(\ref{eqn:form}).   When all the vertical differentials in the
original double complex are zero, there can be no nontrivial spectral
sequence differentials  $d^r$ for $r\geq 2$.  Also, the trivial
vertical differentials mean that there can be no nontrivial extensions
involving anything but the $i$th and $(i+1)$st filtration, but since
we can produce explicit generators whose $p$th powers (in the even
dimensional case) or squares (in the odd dimensional case) actually
vanish, we do not need to worry about extensions at all.    
Thus we obtain the claim about
$\HH_*^{[2],\F_p}(\F_p[x]/x^m)$.

An iteration of this argument yields the result for higher
Hochschild homology, since now we only have exterior algebras on
odd-dimensional classes and truncated algebras, truncated at the
$p$'th power, on even-dimensional ones where the powers that vanish do
so for  combinatorial reasons not relating to the power of $x$ that
was truncated at in the original algebra.  At each stage, the tensor
factor $\F_p[x]/x^m$ will split off the $E^2$-term for 
degree reasons.  What remains will be the $\Tor$ of $\F_p$ with itself over a
differential graded algebra that can be chosen up to chain homotopy
equivalence to be a graded algebra $A$ with a zero differential which
is moreover guaranteed by the flow chart to have the property that
$B(\F_p, A, \F_p)$ is chain homotopy equivalent to its homology
embedded as a subcomplex with trivial differential inside it. 
 
The splitting for higher $\THH$ follows from the splitting for higher
$\HH$ by arguing as in 
\cite[6.1]{blprz} (following \cite[Theorem 7.1]{hm}) for 
the abelian pointed monoid $\{1, x, \ldots, x^{m-1}, x^m=0\}$.
\end{proof}

Reducing the coefficients via the augmentation simplifies things even
further. Here, the result does not depend on the $p$-valuation of $m$,
because $x$ augments to zero and therefore Hochschild
homology of $\F_p[x]/x^m$ with coefficients in $\F_p$ is the homology
of the complex 
$$ \xymatrix@1{
{\ldots} \ar[r]^(0.3){0} & {\Sigma^{m|x|}\F_p}
\ar[rr]^{\Delta(x,x) = 0} & &
{\Sigma^{|x|} \F_p} \ar[r]^{0} & {\F_p.}}$$
Thus we obtain the following result.  
\begin{prop} \label{prop:reduced}
For all primes $p$ and for all $m > 1$ 
$$\HH^{[n],\F_p}_*(\F_p[x]/x^m; \F_p) \cong {B}''_n(\F_p[x]/x^m)$$
where ${B}''_1(\F_p[x]/x^m) \cong \Lambda_{\F_p}(\varepsilon
  x) \otimes \Gamma_{\F_p}(\varphi^0x)$ and
${B}''_n(\F_p[x]/x^m) =
\Tor^{{B}''_{n-1}(\F_p[x]/x^m)}_{*,*}(\F_p, \F_p)$
for $n > 1$. Therefore we obtain 
\begin{equation} \label{eq:splittrred}
\THH^{[n]}_*(\F_p[x]/x^m; \F_p) \cong \THH^{[n]}_*(\F_p) \otimes
{B}''_n(\F_p[x]/x^m).
\end{equation}
\end{prop}

This is shown as in the proof of the previous theorem using the method
of \cite{blprz}, embedding 
$$\varepsilon x\mapsto 1\otimes x\otimes 1,\quad 
\varphi^0x \mapsto 1\otimes x^{m-1} \otimes x\otimes 1$$
inside the bar complex $B(\F_p, \F_p[x]/x^m,\F_p)$, where they
generate exterior and divided power algebras, respectively, regardless
of the divisibility of $m$. 

\begin{rem} \label{rem:pprimetom}
Note that the calculation  becomes drastically different if $(p,m)=1$
and we look at the full $\HH^{[n],\F_p}_*(\F_p[x]/x^m)$ rather than reducing
coefficients to get $\HH^{[n],\F_p}_*(\F_p[x]/x^m;\F_p)$. Then
multiplication by $m$ is an isomorphism on $\F_p[x]/x^m$-modules and
hence 
\[ \HH^{\F_p}_*(\F_p[x]/x^m) \cong 
\begin{cases} \F_p[x]/x^m, & \text{ for } *=0,
  \\
(\Sigma^{|x|(km+1)}\F_p[x]/x^m)/x^{m-1}, & \text{ for } * = 2k+1, \\
\Sigma^{km|x|}\text{ker}(\cdot x^{m-1}), & \text{ for } * = 2k, k > 0.
\end{cases}\]
\end{rem}

\section{A juggling formula}
In this section we generalize juggling formulas from \cite[\S 2 and \S
3]{hhlrz}: we allow working relative to a ring spectrum $R$ that can
be different from the sphere spectrum and we relate the Loday
construction on a suspension on a pointed simplicial set $X$ to the Loday
construction on $X$. In \cite{hhlrz} we mainly considered the cases
where $X$ is a sphere. 

\begin{lem} \label{lem:coequ}
Let $X$ be a pointed simplicial set. For a sequence of cofibrations of
commutative $S$-algebras
\[S \longrightarrow R \longrightarrow A \longrightarrow B \longrightarrow C\]
there is an equivalence of augmented commutative $C$-algebras.
\[\L^A_X(B;C) \simeq C \wedge^L_{\L^R_X(A;C)} \L_X^R(B;C).\]
\end{lem}

\begin{proof}
For the duration of this proof smash products are formed over $R$, not
$S$, but we still denote them by $\wedge$ in order to simplify
notation. For finite $X$ the unit map of $C$, $\eta_C\colon R \ra C$,
induces a map of coequalizer diagrams
$$ 
\xymatrix@C=5em{
\displaystyle{C \wedge \left(\bigwedge_{X_n \setminus *} A\right) \wedge
\left(\bigwedge_{X_n \setminus *} B \right) \ar@<1ex>[d]^{\nu_R}}
\ar@<-1ex>[d]_{\nu_L} \ar[r]^-{\id 
\wedge \eta_C \wedge \eta_C} &  \displaystyle{{C \wedge \left(C \wedge
    \bigwedge_{X_n \setminus *} A\right) \wedge 
\left(C \wedge \bigwedge_{X_n\bs*} B \right)}} \ar@<1ex>[d]^{n_R}
\ar@<-1ex>[d]_{n_L}\\
\displaystyle{{C \wedge \left(\bigwedge_{X_n \setminus *} B \right)}}
\ar[r]^-{\id \wedge \eta_C}  & \displaystyle{C \wedge \left(C \wedge
    \bigwedge_{X_n \bs *} B \right)}
}
$$ 
Here, $\nu_L$ sends the copies of $A$ to $C$ first and then multiplies 
$C \wedge \left(\bigwedge_{X_n \setminus *} C\right)$ to $C$
whereas $\nu_R$ sends the copies of $A$ to $B$ and then uses the
multiplication in each coordinate separately.
The map $n_L$ sends the copies of $A$ to $C$ and multiplies 
$C \wedge \left(C \wedge \bigwedge_{X_n \setminus *} C\right)$ to $C$
sitting at the basepoint 
whereas $n_R$ sends the copies of $A$ to $B$ and then uses the
multiplication in $B$ in every coordinate. This morphism of
coequalizer diagrams induces an isomorphism $\eta$ on the corresponding
coequalizers, \ie, 
$$\eta \colon (\call_X^A(B;C))_n \longrightarrow  (C \wedge_{\call^R_X(A;C)}
\call^R_X(B;C))_n.$$
This is an isomorphism because the diagram
$$ 
\xymatrix@C=5em{
  C
  \ar@<0.5ex>[d]^=
\ar@<-0.5ex>[d]_{=} \ar[r]^-{\id 
\wedge \eta_C \wedge \eta_C} & C\wedge C \wedge C
\ar@<0.5ex>[d]^{\mu\wedge \id} 
\ar@<-0.5ex>[d]_{\id \wedge \mu}\\
C
\ar[r]^-{\id \wedge \eta_C}  & C \wedge C}
$$
induces the identity map between the coequalizers. A colimit argument
proves the claim for general $X$. 
\end{proof}

\begin{thm}[Juggling Formula] \label{thm:juggling}
Let $X$ be a 
pointed simplicial set. Then for any sequence of cofibrations of 
commutative $S$-algebras $S \ra R \ra A \ra B \ra C$ we get an equivalence of 
augmented commutative $C$-algebras
$$ \call_{\Sigma X}^R(B; C) \simeq \call_{\Sigma X}^R(A; C) 
\wedge^L_{\call_X^A(C)} \call_X^B(C).$$
\end{thm}
\begin{proof}
Consider the diagram
$$ \xymatrix{
  {C} & {\call^R_X(A;C)} \ar[r] \ar[l]& {C}\\
  {\call^R_X(C)} \ar[u] \ar[d] & {\call^R_X(A; C)} \ar[r] \ar[l] \ar[u]
  \ar[d] & {C} \ar[u] \ar[d] \\
  {\call^R_X(C)} & {\call^R_X(B;C)} \ar[r] \ar[l] & {C.} }$$ 
By Lemma 
\ref{lem:coequ}, taking the homotopy pushouts of the rows produces the diagram
$$ \xymatrix{
{\call^R_{\Sigma X}(A; C)} \\
{\call_X^A(C)} \ar[u] \ar[d]\\
{\call_X^B(C)}
}$$
whose homotopy pushout is 
$$ \call^R_{\Sigma X}(A; C) \wedge^L_{\call_X^A(C)} \call_X^B(C).$$

We get an equivalent result by first taking the homotopy pushouts 
of the columns and then of the rows. Homotopy pushouts on the columns produces 
$$ \xymatrix{
{C \wedge^L_{\call^R_X(C)} \call^R_X(C)} & 
{\call^R_X(A;C) \wedge^L_{\call^R_X(A;C)} \call^R_X(B;C)} \ar[l] \ar[r] & 
{C \wedge_C^L C}
}$$
which simplifies to
$$ \xymatrix{
{C} & 
{\call^R_X(B;C) } \ar[l] \ar[r] & 
{C}
}$$
whose homotopy pushout is equivalent to $\call^R_{\Sigma X}(B; C)$.
\end{proof}

Restricting our attention to spheres we obtain the following
result. This is a relative variant of \cite[Theorem 3.6]{hhlrz}. 
\begin{cor}\label{cor:juggling}
  Let $S \ra R \ra A \ra B \ra C$ be a sequence of cofibrations of commutative
  $S$-algebras.  Then for all $n \geq 0$ there is an equivalence of 
  augmented commutative $C$-algebras:
$$\THH^{[n+1], R}(B; C) \simeq \THH^{[n+1], R}(A; C) 
\wedge^L_{\THH^{[n], A}(C)} \THH^{[n], B}(C).$$
\end{cor}

\begin{rem}
  The previous corollary gives a splitting of the same form as \cite[Theorem
  3.6]{hhlrz}.  However, as the proof is different it is not obvious that the
  maps in the smash product are the same.  Thus (unlikely as it may be) it may
  turn out to be the case that this gives two different but similar-looking
  splittings.
\end{rem}

\subsection{Beware the phony right-module structure!} \label{subsec:phony}

In some cases it is tempting to use the (valid) splitting of
$\THH^{[n+1],R}(A;C)$ as $\THH^{[n+1],R}(A) \wedge_A C$ and
oversimplify the juggling formula we got in Corollary \ref{cor:juggling} to the
\textbf{invalid} identification  
of $\THH^{[n+1], R}(B; C)$ with $\THH^{[n+1], R}(A) \wedge_A (C
\wedge^L_{\THH^{[n], A}(C)} \THH^{[n], B}(C))$ 
which in the case $B=C$  becomes
\begin{equation} \label{eq:phony} 
\THH^{[n+1], R}(A) \wedge_A \THH^{[n+1], A}(C).
\end{equation}
This transformation is incorrect because it disregards the module
structures, without which the maps of  
pushouts are not well-defined.  The spectrum 
$\THH^{[n+1],R}(A;C)$ is \emph{not}  
equivalent to $\THH^{[n+1],R}(A) \wedge_A C$ as a right-module
spectrum 
over $\THH^{[n], A}(C)$.  
Assuming that the rearrangement that leads to \eqref{eq:phony} were valid,
any cofibration of commutative $S$-algebras
$S \ra A \ra B$ would produce an equivalence between 
$\THH^{[n]}(B)$ and $\THH^{[n]}(A) \wedge^L_A \THH^{[n], A}(B)$. 
But this equivalence does \emph{not} hold in many examples, \eg,  for $A = H\Z$
and $B$ equal
to the Eilenberg Mac Lane spectrum of $\F_p$ or of the ring of
integers in a number field.

\section{A generalization of Brun's spectral sequence} 
In \cite{brun} Morten Brun uses the geometry of the circle to identify
$\THH(HQ; HQ \wedge^L_{Hk} HQ)$ with $\THH(Hk; HQ)$ where $k$ is a
commutative ring and $Q$ is a commutative $k$-algebra: $\THH(Hk; HQ)$
is a circle with $HQ$ at the basepoint and $Hk$ sitting at every
non-basepoint of $S^1$. Homotopy invariance says that we can let the
point take over half the circle, so that it covers an interval. This idea
identifies $\THH(Hk; HQ)$ with $\THH(Hk; B(HQ, HQ, HQ))$ where $B$
denotes the two-sided bar construction. Brun then 
shows in \cite[Lemma 6.2.3]{brun} that the latter is equivalent to
$\THH(HQ; B(HQ, Hk, HQ))$ by a shift of perspective. This idea
inspired our juggling formula \ref{thm:juggling} and also the following result. 
\begin{thm}[Brun Juggling] \label{thm:brun}
  Let $X$ be a pointed simplicial set.  For any sequence of cofibrations of
  commutative $S$-algebras $S \ra R \ra A \ra B$ we get an equivalence of
  commutative $B$-algebras
  \[\L^R_{\Sigma X}(A;B) \simeq B \wedge^L_{\L^R_X(B)} \L^A_X(B).\]
\end{thm}

Note that $B$, which only appears at the basepoint on the left, now appears
almost everywhere on the right.  Thus we can think of the basepoint as having
``eaten'' most of $\Sigma X$.

In the following we will use the notation from \cite[\S 2]{hhlrz}. If
$Y$ is a pointed simplicial subset of $X$, then we denote by
$\L^R_{(X,Y)}(A, B; B)$ the relative Loday construction where we attach
$B$ to every point in $Y$ including the basepoint, $A$ to every point
in the complement and we 
use the structure maps to turn this into a  augmented commutative 
$B$-algebra spectrum. Note that if $Y = *$, then $\L^R_{(X,*)}(A, B; B) =
\L^R_{X}(A; B)$, so in this case we  omit the $*$ from the
notation as in Definition \ref{def:loday}. 
\begin{proof}
  We consider the pair $(\Sigma X, *)$ as $(CX\cup_X CX, CX)$, with the cone
  sitting as the upper half of the suspension.  Then, since the Loday
  construction is homotopy invariant,
  \begin{align*}
  \L^R_{\Sigma X}(A; B) =  \L^R_{\Sigma X, *}(A, B;B) &=
  \L^R_{CX\cup_X CX,*}(A, B;B) 
    \simeq \L^R_{CX\cup_X CX,CX}(A, B;B)  \\ &= \L^R_{CX\cup_X CX, CX\cup_X
      X}(A, B;B).
  \end{align*}
  By \cite[Proposition 2.10(b)]{hhlrz}
\[
  \L^R_{CX\cup_X CX,CX\cup_XX}(A, B;B) 
\simeq \L^R_{CX,CX}(A, B;B) \wedge_{\L^R_{X,X}(A, B;B)} \L^R_{CX,X}(A,
B;B).  \]
By definition $\L^R_{CX,CX}(A, B;B) = \L^R_{CX}(B)$ and $\L^R_X(B) =
\L^R_{X,X}(A, B;B)$ and by homotopy invariance $\L^R_{CX}(B) \simeq B$,
hence 
\[
\L^R_{CX\cup_X CX,CX\cup_XX}(A, B;B) 
\simeq B \wedge_{\L^R_X(B)} \L^R_{CX,X}(A,
B;B).
\]
Using \cite[(3.0.1)]{hhlrz} we can identify $\L^R_{CX,X}(A, B;B)$ with 
\begin{equation} \label{eq:cones}
\L^R_{CX}(A; B) \wedge^L_{\L^R_X(A;B)} \L^R_X(B;B)
\end{equation}
and as $CX$ is contractible we obtain $B \simeq \L^R_{CX}(A; B)$ and
then Lemma \ref{lem:coequ} yields an equivalence of \eqref{eq:cones} with
$\L^A_X(B)$. 
\end{proof}

\begin{ex} \label{ex:brun}
  Consider the case when $X = S^0$.
  \begin{equation} \label{eq:THHeq1}
    \THH(A;B) \simeq B\wedge_{B \wedge B}(B\wedge_A B) = \THH(B; B\wedge_A B).
  \end{equation}
  
  There is an Atiyah--Hirzebruch spectral sequence \cite[IV.3.7]{ekmm}
  \[E^2_{p,q} = \pi_p(E\wedge_R H\pi_qM) \Longrightarrow \pi_{p+q}(E\wedge_R
    M).\] Let $B$ be a connective $A$-algebra.  Setting $R = B \wedge B$,
  $E = B$ and $M = B \wedge_A B$ we get
  \[E^2_{p,q} = \pi_p(B \wedge_{B\wedge B} H \pi_q(B\wedge_A B)) \Longrightarrow
    \pi_{p+q}(B\wedge_{B\wedge B} (B\wedge_A B)).\]
  
  Setting $B = HQ$ and $A = Hk$ gives us
  \[E^2_{p,q} = \THH_p(Q; \Tor_q^k(Q,Q)) \Longrightarrow
  &\pi_{p+q}(HQ\wedge_{HQ\wedge 
      HQ}(HQ\wedge_{Hk} HQ))\\ & \cong \THH_{p+q}(k;Q)\] by
    (\ref{eq:THHeq1}).  This 
  recovers a spectral sequence with the same $E^2$ page and limit as Brun's
  \cite[Theorem 6.2.10]{brun}. A substantial generalization of Brun's spectral
  sequence for $\THH$ can be found in \cite[Theorem 1.1]{hoening}. 
\end{ex}

\begin{ex}
  We can generalize Example~\ref{ex:brun} to any $X$.  In particular, consider a
  commutative ring $k$ and a commutative $k$-algebra $Q$.  If we apply
  the Atiyah--Hirzebruch spectral sequence in the case 
  \[E = HQ  \qquad R = \L_{S^n}(HQ) \qquad M = \L_{S^n}^{Hk}(HQ)\]
  then the Brun juggling formula \ref{thm:brun} gives us a spectral squence 
  \[E^2_{p,q} = \pi_p(HQ \wedge_{\THH^{[n]}(Q)} H\THH_q^{[n],k}(Q))
    \Longrightarrow \THH_{p+q}^{[n+1]}(k;Q).\]
  In the next section, we will see that we can identify
  $\THH^{[n],k}(Q)$ with higher order Shukla homology, $\Sh^{[n],k}(Q)$,
  so we get the simpler description 
  \[E^2_{p,q} = \pi_p(HQ \wedge_{\THH^{[n]}(Q)} H(\Sh_q^{[n],k}(Q))
    \Longrightarrow \THH_{p+q}^{[n+1]}(k;Q).\]
\end{ex}

\begin{cor} \label{cor:augmented}
Let $B$ be an augmented commutative $A$-algebra spectrum. Then applying Theorem \ref{thm:juggling} to
 the sequence 
$\xymatrix@1@C=15pt{S \ar[r]^{=} & S \ar[r] & A \ar[r] & B \ar[r] &  A}$ gives
$$ \call_{\Sigma X}(B; A) \simeq \call_{\Sigma X}(A; A) \wedge^L_A 
\call_X^B(A).$$ 
In particular, if $k$ is a commutative ring,  $A = Hk$,  and $B = HQ$ for
an augmented  commutative $k$-algebra $Q$, then 
$$ \call_{\Sigma X}(HQ; Hk) \simeq \call_{\Sigma X}(Hk; Hk) \wedge^L_{Hk} 
\call_{X}^{HQ}(Hk)$$
and if $k$ is a field, then we obtain on the level of homotopy groups 
$$ \pi_*\call_{\Sigma X}(HQ; Hk) \cong \pi_*(\Sigma X \otimes Hk)
\otimes_k \pi_*(\call_X^{HQ}(Hk)).$$
\end{cor}
\begin{rem}
We stress that in Corollary \ref{cor:augmented} there is a spectrum
level splitting of $\call_{\Sigma X}(HQ; Hk)$ into $\call_{\Sigma
  X}(Hk)$ smashed with an additional factor. In particular, for
$X=S^n$ higher $\THH$ of an augmented commutative $k$-algebra splits
as 
$$ \THH^{[n+1]}(Q;k) \simeq \THH^{[n+1]}(k) \wedge^L_{Hk} \THH^{[n],Q}(k).$$
Greenlees proposed a splitting result in \cite[Remark 7.2]{greenlees}:
If $k$ is a field and $Q$ is a augmented commutative  $k$-algebra, then
his results yield a splitting 
$$ \THH_*(Q; k) \simeq \THH_*(k) \otimes_k \Tor_*^Q(k,k).$$
Our result generalizes his because for $X = S^0$ the term
$\call_{S^0}^{HQ}(Hk)$ is nothing but $Hk \wedge^L_{HQ} Hk$ whose
homotopy groups are isomorphic to $\Tor_*^Q(k,k)$. We will revisit
this splitting result later in Theorem \ref{prop:genGreenlees},
relating it to higher order Hochschild homology.  
\end{rem}

\section{Higher Shukla homology} \label{sec:shukla}
Let $k$ be a commutative ring. Ordinary Shukla homology \cite{shukla}
of a $k$-algebra $A$ with coefficients in an  
$A$-bimodule $M$ can be identified with $\THH^{k}(A; M)$. We will define
higher order Shukla homology in the context of commutative algebras as
an iterated bar construction and 
identify it with $\THH^{[n],k}(A; M)$ in Proposition \ref{prop:same}. 

\begin{defn}
  Let $A$ be a commutative $k$-algebra and $B$ be a commutative $A$-algebra.  We
  define
  \[\Sh^{[0],k}(A;B) = HA \wedge^L_{Hk} HB.\]
  For $n \geq 1$ we define \emph{$n$th order Shukla homology of $A$ over $k$
    with coefficients in $B$} as
  $$ \Shukla^{[n],k}(A;B) = B^S(HB, \Shukla^{[n-1],k}(A; B), HB). $$
  where the latter is the two sided bar construction with respect to $HB$ over
  the sphere spectrum.
\end{defn}

Thus for $n = 1$ we have $\Sh^{[1],k}(A;B) \simeq \THH^{k}(A;B)$.  For example,
when $k = \Z$, $p$ is a prime and $a = p^m$,
$$\Shukla^{\Z}_*(\Z/p^m; \Z/p) \cong \Gamma_{\Z/p}(x(m))$$
with $|x(m)| = 2$.

It is consistent to set $\Shukla^{[-1],k}(A) =Hk$. 

A priori, $\THH^{k}(A; B)$ is a simplicial spectrum and
$\Shukla^{[n],k}(A;B)$ is therefore an $n$-simplicial spectrum, but we
take iterated diagonals to get a simplicial spectrum and can then use
geometric realization to get an honest spectrum.

\begin{prop} \label{prop:shuklazzp} \label{lem:shTor} Let $R$ be a commutative
  ring and let $a,p\in R$ be elements which are not zero divisors such
  that $(p)$ is maximal and $a \in
  (p)^2$.
  Then
  \[\Sh^{[0],R}_*(R/p)  \cong \Lambda_{R/p}(\tau_1) \cong
  \Sh^{[0],R}_*(R/a;R/p), \qquad |\tau_1| = 1\] 
  and for $n \geq 2$,
  \[\Sh_*^{[n], R}(R/p) \cong \Tor_*^{\Sh_*^{[n-1],R}(R/p)}(R/p;R/p)\]
  and
  \[\Sh_*^{[n],R}(R/a;R/p) \cong \Tor_*^{\Sh_*^{[n-1],R}(R/a,R/p)}(R/p;R/p).\]
\end{prop}

\noindent
\textbf{Warning:} the reduction $R/a \rto R/p$ does {\Large
  \textbf{not}} induce an isomorphism $\Sh_*^{[n],R}(R/a;R/p) \rto
\Sh_*^{[n],R}(R/p)$.  By considering resolutions we can see that at
$n=0$ the induced map is the map taking $\tau_1$ to $0$.  In fact, in
Corollary~\ref{cor:sh:ap->pp} we show that the map induced by $R/a
\rto R/p$ is zero on all generators other than the $R/p$ in dimension $0$.

\begin{proof}
  We prove this by induction on $n$.  At $n = 0$,
  \[\Sh^{[0], R} (R/p; R/p) = R/p \wedge^L_{R} R/p.\]
  There is a K\"unneth spectral sequence,
  \[E^2_{s,t} = \Tor_{s,t}^R (R/p, R/p) \Longrightarrow \pi_{s+t}(R/p \wedge^L_{R}
    R/p).\]
  We have a short resolution
  \[R \xra{\cdot p} R \to R/p,\]
  so
  \[\Tor^R_{s,t}(R/p,R/p) \cong H_s(R/p \xra{0} R/p)_t \cong
    \begin{cases}
      R/p, & s = 0 = t, \\
      R/p, & s = 1, t = 0, \\
      0, & \text{ otherwise.}
    \end{cases}\]
  For degree reasons, there cannot be any differentials or
  extensions in this spectral sequence, and the product of $\tau_1$
  with itself has to vanish.  Thus $\Sh^{[0]}_*(R/p) \cong
  \Lambda_{R/p}(\tau_1)$, as 
  desired.  Note that this proof works (almost) verbatim for $\Sh_*^{[0],
    R}(R/a;R/p)$.

  By \cite[Proposition 2.1]{dlr}, as a  augmented commutative $HR/p$-algebra,
  \[\Sh^{[0],R}(R/p) \simeq HR/p \vee \Sigma HR/p \simeq
    \Sh^{[0],R}(R/a;R/p).\]
  Thus
  \[\Sh^{[1],R}(R/p) = B(HR/p, \Sh^{[0],R}(R/p), HR/p) \simeq HB(R/p,
    \Lambda_{R/p}(\tau_1), R/p).\]

In the following let $\F$ be $R/p$. 
 By \cite{blprz}, if we start with $B^{\F}(\F, \Lambda_\F(\tau_1),
 \F)$  for $\F$ a field of positive characteristic, we know that the
 spectral sequence $\Tor_{*,*}^{ \Lambda_\F(\tau_1)}(\F,\F)
 \Rightarrow H_*(B^{\F}(\F, \Lambda_\F(\tau_1), \F))$ collapses at
 $E^2$, which concludes the proof of the $n=1$ case.   We have that
 $B^{\F}(\F, \Lambda_\F(\tau_1), \F)\simeq B(\F, \Lambda_\F(\tau_1),
 \F)$ because both calculate the homology of $\F\otimes^L _{
   \Lambda_\F(\tau_1)} \F$.  Moreover, in  \cite{blprz} we show that
 if we keep applying $B^{\F}(\F, -, \F)$ to the result, having started
 with 
 $ \Lambda_\F(\tau_1)$, the spectral sequences
 $\Tor_{*,*}^{ H_*(-)}(\F,\F) \Rightarrow H_*(B^{\F}(\F,-, \F))$ will
 keep collapsing.  Since in the case of a characteristic zero field, a
 divided power algebra is isomorphic to a polynomial one, we can use
 the method of \cite{blprz}, adjusted as in the proof of Proposition
 \ref{prop:flowcharts}, to get the same result. 
  
 Finally, in \cite{dlr} we show that once we can exhibit a commutative
 $H\F$-algebra as the image of the Eilenberg Mac Lane functor on some
 simplicial algebra, we can continue doing that when we apply
 $B(\F,-,\F)$ to that algebra---once we get to the algebraic setting
 we can stay there.  This concludes the proof for the collapsing of
 the spectral sequences both for $R/p$ and for $R/a$. 
 \end{proof}

\begin{defn}
For any commutative $k$-algebra  $A$ and any commutative $A$-algebra $B$ we
define \emph{higher derived Hochschild homology of $A$ over $k$ with
  coefficients in $B$}, 
$\widetilde{\HH}^{[n],k}(A; B)$,  as
$$  \widetilde{\HH}^{[n],k}(A; B) = \THH^{[n], k}(A; B). $$
\end{defn} 
Note that $\Shukla^{[1],k}(A;B) =  \widetilde{\HH}^{[1],k}(A; B) =
\Shukla^k(A; B)$. 

\begin{prop}\label{prop:same}
There is an isomorphism 
$$ \Shukla^{[n],k}_*(A; B) \cong \widetilde{\HH}^{[n],k}_*(A; B)$$
for all $n \geq 0$. 
\end{prop}
The idea for the following proof is due to Bj{\o}rn Dundas. 
\begin{proof}
By definition, the claim is true for $n=0$. For higher $n$ we
have to show that  
the two-sided bar construction $B^S(HB, \THH^{[n], k}(A; B), HB)$
is a model for  
$\THH^{[n+1], k}(A; B)$; then the claim follows by
induction. Using the decomposition of  
the $(n+1)$-sphere as two hemispheres glued along the equator
$S^n$ gives a homotopy pushout  
diagram for $\THH^{[n+1], k}(A; B)$ 
$$ \xymatrix{
{\THH^{[n], k}(A; B)} \ar[d] \ar[r] & {HB} \ar[d] \\
{HB} \ar[r] & {\THH^{[n+1], k}(A; B)}
}$$
in the model category of commutative $Hk$-algebras. The latter
category is equivalent to the  
category of commutative $S$-algebras under $Hk$. The two-sided bar
construction  
$$B^S(HB, \THH^{[n], k}(A; B), HB)$$ 
models the homotopy pushout 
$$ HB \wedge^L_{\THH^{[n], k}(A; B)} HB$$ 
in the category of $S$-modules but this is also the homotopy pushout
in the category  
of commutative $S$-algebras and in the category of commutative
$S$-algebras under $Hk$; here one can use the model structure from
\cite{ekmm}   
where cofibrant commutative $S$-algebras give the correct homotopy
type when involved in a  smash
product as an underlying $S$-module.   
\end{proof}

\begin{prop} \label{prop:stablesh}
  In the special case of a sequence of cofibrations of commutative $S$-algebras
  $R=Hk \ra Hk \ra HA \ra Hk$ with a cofibrant model of $Hk$ and an augmented
  commutative $k$-algebra $A$ we obtain
\begin{equation} \label{eq:xshukla}
\call_{\Sigma X}^{Hk}(HA; Hk) \simeq \call^{HA}_X(Hk)
\end{equation} 
for any $X$.   
\end{prop}
\begin{proof}
The juggling formula \ref{thm:juggling} for the sequence $Hk \ra Hk
\ra HA \ra Hk$ gives  
$$ \call_{\Sigma X}^{Hk}(HA; Hk) \simeq \call_{\Sigma X}^{Hk}(Hk; Hk) 
\wedge^L_{\call_X^{Hk}(Hk)} \call_X^{HA}(Hk)$$
but $\call_X^{Hk}(Hk) \simeq Hk \simeq \call_{\Sigma X}^{Hk}(Hk; Hk)$.
\end{proof}
\begin{rem}
Note that Proposition \ref{prop:stablesh} implies that $\call^{HA}_X(Hk)$
depends only on the homotopy type of $\Sigma X$, so $\call^{HA}_X(Hk)$
is a stable invariant of $X$. 
\end{rem} 
Consider the case where $A$ is flat over $k$ and $X = S^n$. Then
Equation \eqref{eq:xshukla} gives
\begin{equation} \label{eq:snshukla}
\call_{S^{n+1}}^{Hk}(HA; Hk) \simeq
\call^{HA}_{S^n}(Hk). 
\end{equation}
The term on the left hand side of \eqref{eq:snshukla} has as homotopy groups the
Hochschild homology of order $n+1$ of $A$ with coefficients in $k$. The right
hand side simplifies to $ \THH^{[n], A}(k)$ and this is Shukla homology of
order $n$ of $k$ over $A$. Therefore we obtain:
\begin{prop} \label{prop:hhsh}
Let $k$ be a commutative ring and let $A$ be an augmented commutative 
$k$-algebra which is flat as a $k$-module. Then for all $n \geq 0$
$$ \HH^{[n+1],k}_*(A; k) \cong \Sh^{[n], A}_*(k).$$
\end{prop}
Note that for $n=0$ we obtain the classical formula \cite[X.2.1]{CE}
$$ \HH^{k}_*(A;k) \cong \Tor^A_*(k, k)$$
Combining Proposition \ref{prop:hhsh} with Corollary
\ref{cor:augmented} we get the following splitting result
for augmented commutative $k$-algebras. 
\begin{thm} \label{prop:genGreenlees}
Let $k$ be a commutative ring and let $A$ be an augmented commutative 
$k$-algebra which is flat as a $k$-module. Then for all $n \geq 0$ 
$$ \THH^{[n]}(A; k) \simeq \THH^{[n]}(k) \wedge_{Hk}  \THH^{[n],k}(A; k).$$
If $k$ is a field then we obtain the following isomorphism on the
level of homotopy groups 
$$ \THH_*^{[n]}(A; k) \cong \THH^{[n]}_*(k) \otimes_{k}  \HH_*^{[n],k}(A; k).$$
\end{thm}

\section{A weak splitting for $\THH(R/(a_1,\ldots, a_r); R/\m)$}

Using a Tor-calculation by Tate from the 50's we obtain a splitting on
the level of homotopy groups of $\THH_*(R/(a_1,\ldots, a_r); R/\m)$ in
good cases. This yields an easy way of calculating $\THH_*(\Z/p^m;
\Z/p)$ for $m \geq 2$. Compare
\cite{pirashvili,brun,angeltveit} for other approaches. 

\begin{thm} \label{thm:Tate}
Let $R$ be a regular local ring with maximal ideal $\m$ and let $(a_1, \ldots, 
a_r)$ be a regular sequence in $R$ with $a_i \in \m^2$ for $1 \leq i
\leq r$. Then 
$$ \THH_*(R/(a_1,\ldots, a_r); R/\m) \cong \THH_*(R; R/\m) \otimes_{R/\m}
\Gamma_{R/\m}(S_1, \ldots, S_r)$$ 
with $|S_i| =2$. 
\end{thm}
\begin{proof}
  Let $I = (a_1,\ldots,a_r)$.  Applying the juggling formula
  \ref{thm:juggling} to $X = S^0$ and to
  the sequence $S \ra HR \ra HR/I \ra HR/\m$ gives
  $$ \THH(R/I; R/\m) \simeq \THH(R; R/\m) 
  \underset{HR/\m \wedge^L_{HR} HR/\m}{\wedge^L}
 ( HR/\m \underset{{HR/I}}{\wedge^L} HR/\m).$$
  In \cite{Tate} Tate determines the algebra structure on the homotopy
  groups of the last term, 
  $$ \Tor_*^{R/I}(R/\m, R/\m) \cong 
  \Lambda_{R/\m}(T_1, \ldots, T_d) \otimes_{R/\m} 
  \Gamma_{R/\m}(S_1, \ldots, S_r).$$
  Here, $d$ is the dimension of $\m/\m^2$ as an $R/\m$-vector space. We 
  can choose a regular system of generators $(t_1, \ldots, t_d)$ for $\m$ 
  such that the module structure of $\Tor_*^{R/I}(R/\m, R/\m)$ 
  over $\Tor_*^R(R/\m, R/\m) \cong \Lambda_{R/\m}(T_1, \ldots, T_d)$ is the 
  canonical one (see \cite[p.~22]{Tate}). Hence the K\"unneth spectral sequence 
  for $\THH(R/I; R/\m)$ has an $E^2$-term isomorphic to 
  $$ \THH_*(R; R/\m) \otimes_{R/\m} \Gamma_{R/\m}(S_1, \ldots, S_r)$$
  which is concentrated in the zeroth column that consists of 
  $$\THH_*(R; R/\m) \otimes _{ \Tor_*^{R}(R/\m, R/\m)}
  \Tor_*^{R/I}(R/\m, R/\m).$$ Therefore, there are no  
  non-trivial differentials and extensions in this spectral sequence. 
\end{proof}

We call the splitting of Theorem \ref{thm:Tate} a \emph{weak
  splitting} because it is only a splitting on the level of homotopy
groups. In Section~\ref{sec:splitting} we develop a stronger spectrum-level
splitting of a similar form.

We apply the above result in the special case where $R$ is a principal 
ideal domain. 
Let $p \neq 0$ be an element of $R$, such that $(p)$ is a maximal ideal in 
$R$ and let $n$ be bigger or equal to $2$. Then  we are in the situation of 
Theorem \ref{thm:Tate} because $R_{(p)}/(p^n) \cong R/(p^n)$ so we can drop 
the assumption that $R$ is local. 
The above result immediately gives an explicit formula for 
$\THH(R/(p^n); R/(p))$. 
\begin{cor} \label{cor:thhpid}
For all $n > 1$: 
$$ \THH_*(R/(p^n); R/(p)) \cong \THH_*(R; R/(p)) \otimes_{R/(p)} 
\Gamma_{R/(p)}(S_1).$$
\end{cor}

\begin{rem}
One may try to use the same method for
$\THH^{[n]}$.  The juggling formula from Theorem \ref{thm:juggling}
for $\xymatrix@1@C=15pt{S \ar[r]^{=} & S \ar[r] & H\Z \ar[r] & H\Z/p^m \ar[r]
  & H\Z/p}$ gives us 
\[\THH^{[n]}(\Z/p^m; \Z/p) \simeq \THH^{[n]}(\Z;\Z/p)
  \underset{\Sh^{[n-1],\Z}(\Z/p)}{\wedge^L} \Sh^{[n-1],\Z/p^m}(\Z/p).\]
Thus we must understand the structure of $\Sh^{[n-1],\Z/p^m}(\Z/p)$ as a
$\Sh^{[n-1],\Z}(\Z/p)$-algebra.  It is not possible to do this through direct
$\Tor$ computations for all $n$, as the computations rapidly become
intractable; even 
$\Sh^{[1],\Z/p^2}(\Z/p)$ is rather involved \cite[(5.2)]{BP}, but see
Proposition \ref{prop:shuklaRaRp} for a general formula.  

In order to obtain calculations in this example and in related cases, we need to
develop the more delicate splitting of 
Section~\ref{sec:splitting}.  
\end{rem}

\section{A splitting for $\THH^{[n]}(R/a;R/p)$}
\label{sec:splitting}

Throughout this section, we assume that $R$ is a commutative
  ring and $a,p\in R$ are elements which are not zero divisors for which
   $(p)$ is a maximal ideal and $a \in
  (p)^2$. 

\begin{lem} \label{lem:sh:ap->pp}
  Let $R$, $p$, and $a$ be as above, and let $\pi:R/a \rto R/p$ be the
  obvious reduction. 
   Then the map induced by  $\pi$, 
  \[\pi_*:\Sh_*^{[0],R}(R/a;R/p) \longrightarrow \Sh_*^{[0],R}(R/p),\]
  factors as 
$$ \xymatrix@1{\Sh_*^{[0],R}(R/a;R/p) \ar[r]^(0.7)\epsilon & R/p
  \ar[r]^(0.3)\eta &  \Sh_*^{[0],R}(R/p).} $$
\end{lem}

\begin{proof}
  The assumptions on $a$ and $p$ ensure that there exists a $b\in R$
  such that $a = bp^2$. 
  We have the following diagram:
  \[\xymatrix{
      R \ar[r]^{\cdot a}  \ar[d]_{\cdot bp} & R \ar[r]^-\epsilon \ar[d]^= & R/a \\
      R \ar[r]^{\cdot p} & R \ar[r]^-\epsilon & R/p
    }\]
  Thus we have a map of resolutions.  When we tensor up with $R/p$ we get the
  following diagram:
  \[\xymatrix{
      R \otimes_R R/p \ar[r]^{a\otimes 1} \ar[d]_{bp\otimes 1 = 0} &
      R\otimes_R R/p  \ar[d]^=\\
      R \otimes_R R/p \ar[r]^{p\otimes 1} & R\otimes_R R/p 
      }
    \]
    We take the homology of the top and bottom row.  Note that since $a,p \in
    (p)$, the horizontal maps are $0$; thus the top and bottom row produce
    $\Tor$'s which are of the form $\Lambda_{R/p}(\tau_1)$.  However, when we
    look at where $\tau_1$ goes from the top to the bottom, it maps by
    multiplication by $bp$---which is $0$ in $R/p$.  Thus this map is $0$.
\end{proof}

Surprisingly enough, this special case allows us to prove a spectrum-level
splitting for all $n \geq 0$.

\begin{defn}
  Let $\A_{HR/p}$ be the category of augmented commutative
  $HR/p$-algebras and $h\A_{HR/p}$ its homotopy category.  Let
  $\Mod_{HR/p}$ be the category of $HR/p$-modules.
\end{defn}

\begin{lem} \label{lem:ap->pp}
For $R$, $p$, and $a$  as above, the map
  \[\varphi_n\colon\THH^{[n], R}(R/a; R/p) \rto \THH^{[n], R}(R/p)\]
  induced by $R/a \rto R/p$ factors through $HR/p$ in $h\A_{HR/p}$.
\end{lem}

\begin{proof} 
  The key step is the $n = 0$ case.

  From \cite[Proposition 2.1]{dlr} we know that
  $\THH^{[0],R}(R/a; R/p) \simeq HR/p \vee \Sigma HR/p$ and also
  $\THH^{[0],R}(R/p) \simeq HR/p \vee \Sigma HR/p$.  So we need to
  understand 
  \[h\A_{HR/p}(H R/p \vee \Sigma HR/p, HR/p \vee \Sigma
    HR/p).\]
  By \cite[Proposition 3.2]{basterra}, we can identify this as
  \[h\Mod_{HR/p}(\LL Q\RR I(HR/p \vee \Sigma HR/p), \Sigma HR/p).\]
  Given an $A\in \A_{HR/p}$, we have a pullback
  \[\xymatrix{
      IA\ar[d] \ar[r] & A \ar[d]^\epsilon \\
      {*} \ar[r] & HR/p
    }\]
  Let $B$ be a nonunital $HR/p$-algebra. Basterra defines $Q(B)$ to be the
  pushout
  \[\xymatrix{
      B \wedge_{HR/p} B \ar[r] \ar[d] & B \ar[d] \\
      {*} \ar[r] & Q(B)
    }\]
  We want to take the left- and right-derived versions of these functors for
  Basterra's result.  

  Let $X$ be a fibrant replacement for $HR/p \vee \Sigma HR/p$ in $\A_{HR/p}$,
  so that $X \xra\epsilon HR/p$ is a fibration.  Thus the square
  \[\xymatrix{\RR I(HR/p\vee \Sigma HR/p) \ar[r] \ar[d] & X
    \ar@{->>}[d]^\epsilon \\ 
      {*} \ar[r] & HR/p}\]
  is a homotopy pullback square (since every spectrum is fibrant
  \cite[Proposition 13.1.2]{hirschhorn}).  For conciseness we write
  $Y = \RR I(HR/p\vee \Sigma HR/p)$.  We have a long exact sequence of homotopy
  groups
  \[0 \rto \pi_1 Y \rto \pi_1 X \rto \pi_1 HR/p \rto \pi_0 Y \rto \pi_0 X \rto
    \pi_0 HR/p \rto \pi_{-1} Y \rto 0,\]
  where we have used that $X \simeq HR/p \vee \Sigma HR/p$ so that its homotopy
  groups are concentrated in degrees $0$ and $1$.  Note that the map $\pi_0X
  \rto \pi_0 HR/p$ is the identity.  We thus see that $\pi_iY \cong 0$ for $i
  \neq 1$ and $\pi_1Y \cong R/p$. 

  We need to identify
  \[h\Mod_{ HR/p}(\LL Q(\Sigma  HR/p), \Sigma HR/p) \cong \pi_0 F_{
      HR/p}(\LL Q(\Sigma HR/p), \Sigma  HR/p),\]
  where $F_{HR/p}(\cdot,\cdot)$ is the function spectrum.
 We use the universal coefficient spectral sequence
  \[E^{s,t}_2 &= \Ext_{R/p}^{s,t}(\pi_*\LL Q(\Sigma HR/p), \pi_*\Sigma HR/p) \\
    &\Longrightarrow \pi_{t-s} F_{ HR/p}(\LL Q(\Sigma HR/p), \Sigma HR/p).\]
  Note that we're working over a field, so $E_2^{s,t} = 0$ for $s \neq 0$, and
  $\pi_*(\Sigma HR/p)$ is zero everywhere except at $\pi_1$, so in fact the
  spectral sequence collapses at $E_2$.  Since we are only interested in
  $\pi_0$, the only term relevant to us is
  \[E_\infty^{0,0} \cong E_2^{0,0}
  \cong \Hom_{R/p}(\pi_1\LL Q(\Sigma HR/p), R/p).\]  Note that this group
  cannot be $0$, since our hom-set contains at least two elements: the identity
  map and the $0$ map. Thus it remains to compute $\pi_1 \LL Q(\Sigma HR/p)$.
  
  Consider the diagram
  \[\xymatrix{
      (\Sigma HR/p)^{cof} \ar[r]^f \ar@{->>}[d]_\sim & \LL Q(\Sigma HR/p) \\
      \Sigma HR/p 
    }\]
  By \cite[Proposition 2.1]{BastMcC}, since $(\Sigma HR/p)^{cof}$ is
  $0$-connected, $f$ is $1$-connected; thus $\pi_1f$ is surjective.  Since $R/p$
  is a field, we must have
  \[\pi_1\LL Q(\Sigma HR/p) \cong 0 \hbox{ or } R/p.\]
  Since it can't be $0$, it must be $R/p$, with the induced map being the
  identity.  

  Let $\tau_1^{(a)}$ be the generator of the $\Lambda(\tau_1)$ obtained as
  $\Sh^{[0],R}(R/a;R/p)$ and let $\tau_1^{(p)}$ be the generator of the
  $\Lambda(\tau_1)$ obtained as $\Sh^{[0],R}(R/p)$ in the calculation of
  Lemma~\ref{lem:sh:ap->pp}.   The above calculation shows
  that
  \[h\Mod_{H R/p}(\LL Q \RR I(HR/p \vee \Sigma HR/p),\Sigma HR/p) \\ \qquad\ 
    \qquad\ \qquad \cong
    \Hom_{R/p}(R/p, R/p)\]
where the first copy of $R/p$ is generated by $\tau_1^{(a)}$ and the
second copy is generated by $\tau_1^{(p)}$. But the induced map on
  $\Sh^{[0],R}$ takes $\tau_1^{(a)}$ to $0$.  Thus the corresponding map in
  $h\A_{R/p}$ is also $0$.  This proves the $n=0$ case. 
  
  We now turn to the induction step.  We have the composition
  \[\xymatrix@C=-1.5em{
      B(HR/p, \THH^{[n-1],R}(R/a;R/p), HR/p) \ar[rd]^{\mathrm{hyp.}}
      \ar[dd]_{B(1,\varphi_{n-1},1)}
      \\
      &  B(HR/p, HR/p, HR/p)   \ar[ld]  \\
       B(HR/p, \THH^{[n-1],R}(R/p; R/p), HR/p)
    }\]
  of maps of simplicial spectra.  Taking the realization gives us the
  composition 
  \[\varphi_n: \THH^{[n]}(R/a;R/p) \longrightarrow HR/p \longrightarrow
    \THH^{[n]}(R/p),\]
  as desired.
\end{proof}

By applying $\pi_*$ to the result of Lemma~\ref{lem:ap->pp} we get the following
generalization of Lemma~\ref{lem:sh:ap->pp}.

\begin{cor} \label{cor:sh:ap->pp}
  For $R$, $p$, and $a$ as above, for all $n \geq 0$ the map
  \[\Sh^{[n], R}_*(R/a; R/p) \longrightarrow \Sh_*^{[n],R}(R/p)\]
  induced by $R/a \rto R/p$ factors as
  \[\xymatrix@1{\Sh^{[n],R}_*(R/a; R/p) \ar[r]^(0.7)\epsilon & R/p
    \ar[r]^(0.3)\eta & \Sh_*^{[n],R}(R/p).}\] 
  Here, $R/p$ is considered as a graded ring concentrated in degree $0$; the
  first map in the factorization is the augmentation and the second is the unit
  map induced by the inclusion of the basepoint.
\end{cor}

By Lemma \ref{lem:coequ}, 
\[ \THH^{[n],R/a}(R/p) \simeq HR/p \wedge_{\THH^{[n],R}(R/a;R/p)}
  \THH^{[n],R}(R/p). \]
However, by Lemma~\ref{lem:ap->pp} the map $\THH^{[n],R}(R/a;R/p) \rto
\THH^{[n],R}(R/p)$ factors through $HR/p$. This proves the following
result about higher order Shukla homology: 
\begin{prop} \label{prop:shuklaRaRp} For $R$, $p$, and $a$ as above,
\begin{align*} 
\THH^{[n],R/a}(R/p) & \simeq 
 (HR/p \wedge_{\THH^{[n],R}(R/a;R/p)} \wedge HR/p)  \wedge_{HR/p}
  \THH^{[n],R}(R/p) \\ 
 & \simeq \THH^{[n+1],R}(R/a;R/p) \wedge_{R/p}
  \THH^{[n],R}(R/p).
\end{align*}
\end{prop}

This recovers the calculation of $\Sh_*^{\Z/p^2}(\Z/p)$ from
\cite[5.2]{BP}. It also explains why these Shukla calculations are
more involved than Shukla homology calculations of the form
$\Sh_*^R(R/x)$ where $x$ is a regular element. In the latter case we
just obtain a divided power algebra over $R/x$ on a generator of
degree two, whereas for all $m \geq 2$ 
\begin{align*}
\Sh_*^{\Z/p^m}(\Z/p) & \cong \Sh^{[2],\Z}_*(\Z/p^m; \Z/p)
\otimes_{\Z/p} \Sh^{\Z}_*(\Z/p) \\
& \cong \bigotimes_{i \geq 0}
\Lambda(\varepsilon(\varrho^i(\tau_1^{(p^m)}))) 
\otimes \Gamma_{\Z/p}(\varphi^0\varrho^i(\tau_1^{(p^m)})))
\otimes_{\Z/p} \Gamma_{\Z/p}(\varrho^0(\tau_1^{(p)})).    
\end{align*}

We are now ready to prove the main splitting result:

\begin{thm} \label{thm:Zp^m} If $R$ is a commutative ring and if $p,a\in R$ are 
elements which are not zero divisors for which $(p)$ is a maximal ideal
and $a\in(p)^2$, then 
  \[\THH^{[n]}(R/a; R/p) \simeq \THH^{[n]}(R; R/p) \wedge^L_{HR/p} \THH^{[n],
      R}(R/a;R/p).\]
\end{thm}

\begin{proof}
Recall that in the category of commutative algebras, the smash product
is the same as the pushout. Consider the following diagram: 
\[\xymatrix{
HR/p \ar[rr] \ar[d] & & \THH^{[n-1],R}(R/p;R/p) \ar[rr]\ar[d]
      & & \THH^{[n]}(R;R/p) \ar[d] \\
\THH^{[n],R}(R/a;R/p) \ar[rr] && \THH^{[n-1],R/a}(R/p;R/p)
      \ar[rr] && \THH^{[n]}(R/a;R/p). }\] 
By Proposition \ref{prop:shuklaRaRp} the left square is a homotopy
pushout square and the right square is a homotopy pushout 
square by the juggling formula \ref{thm:juggling}, with the maps of
those formulas. Thus the  outside of the diagram also gives a homotopy pushout,
producing the formula 
  \[\THH^{[n]}(R/a;R/p) \simeq \THH^{[n]}(R;R/p)
    \wedge^L_{HR/p} \THH^{[n],R}(R/a;R/p),\]
as desired.
\end{proof}

\section{The Loday construction of products}
\label{sec:sums}
We establish a splitting formula for Loday constructions of products of
ring spectra. This result is probably well-known,
but as we will need it later, we provide a proof.  In the case of $X =
S^1$ such a splitting is proved for connective ring spectra in 
\cite{dmcc} in the context of 'ring functors'. See also
\cite[Proposition 4.2.4.4]{dgm}.

For two ring spectra $A$ and $B$ we consider their product $A
\times B$ with the multiplication 
\[ (A \times B) \wedge (A \times B) \ra  A \times B \]
that is induced by the maps 
$(A \times B) \wedge (A \times B) \ra  A$ and $(A \times B) \wedge
(A \times B) \ra  B$ that are given by the projection maps to $A$
and $B$ and the multiplication on $A$ and $B$:

\[
  \xymatrix{(A \times B) \wedge (A \times B) \ar[rr]^(0.6){\text{pr}_A
      \wedge \text{pr}_A} \ar[d]_{\text{pr}_B \wedge \text{pr}_B} & &
    A \wedge A \ar[r]^{\mu_A}& A \\ 
    B \wedge B \ar[d]_{\mu_B}& & & \\
  B & & &}
\]
For $X= S^0$ we obtain  
$$\L_{S^0}(A \times B) = (A \times B) \wedge (A \times B)$$ 
and this is equivalent to $A \wedge A \vee A \wedge B \vee B
\wedge A \vee  
B \wedge B$ whereas $\L_{S^0}(A) \times \L_{S^0}(B)$ 
is equivalent to $A \wedge A \vee B \wedge B$ so in this case 
$\L_{S^0}(A \times B)$ is not equivalent to
$\L_{S^0}(A) \vee \L_{S^0}(B)$. In general, if 
a simplicial set has finitely many connected components, say $X = X_1
\sqcup \ldots \sqcup X_n$, then 
$$
\L_X(A \times B)  \simeq \L_{X_1}(A \times B) \wedge \ldots
\wedge \L_{X_n}(A \times B)$$ 
so it suffices to study $\L_X(A \times B)$ for connected simplicial
sets $X$.  We will
first consider the case $X=S^1$, where $\L_{S^1}$ with respect to
the minimal simplicial model of the circle is $\THH=\THH^{[1]}$, and
then use that special case to prove the result for general connected 
finite simplicial sets $X$.

We thank Mike Mandell who suggested to use Brooke Shipley's version of $\THH$
in the setting of symmetric spectra. Brooke Shipley shows in
\cite{shipley} that a variant of B\"okstedt's model 
for $\THH$ in symmetric spectra of simplicial sets
is equivalent to the version that mimics the Hochschild complex and she
proves several important features of this construction. 
See also \cite{ps} for a correction of the proof of the comparison.

\begin{prop}\label{prop:THHsum}
  For symmetric ring spectra $A$ and $B$, the product of the projections induces a stable
  equivalence
 $\THH(A \times B) \to \THH(A) \times \THH(B)$. 
\end{prop}

In the following proof we denote by $\THH$ the model of $\THH$ defined in
\cite[Definition 4.2.6]{shipley}. This is no abuse of notation:
Let $A$ and $B$ be $S$-cofibrant symmetric ring
spectra \cite[Theorem 2.6]{shipley-conv}. Then by
\cite[Theorem 4.2.8]{shipley} and \cite[Theorem 3.6]{ps}
$\THH(A)$ and $\THH(B)$ are stably equivalent to $\THH(A)$ and $\THH(B)$ in
our sense. As $\THH(-)$ sends stable equivalences to stable equivalences
(\cite[Corollary 4.2.9]{shipley}) we can choose an $S$-cofibrant 
replacement of $A \times B$, $(A \times B)^c$, and get that $\THH(A \times B)$
is stably equivalent to $\THH((A \times B)^c)$ and this in turn is stably
equivalent to our notion of $\THH$.

\begin{proof}
Note that for any symmetric ring spectrum $R$, $\THH(R)$ is defined as the
diagonal of a bisimplicial symmetric spectrum $\THH_\bullet(R)$
\cite[4.2.6]{shipley}, where one of the simplicial directions 
comes from the $\THH$-construction and the other one comes from the fact
that we are working with symmetric spectra in simplicial sets. In
\cite[p.~4101]{ps} the authors use the geometric realization instead of the diagonal, but this
does not cause any difference in the arguments. 

We will start by showing that there is a chain of stable equivalences
  between  $\THH(A \times B)$ and $\THH(A) \times \THH(B)$. 
We use the following chain of identifications:
\[\xymatrix@C=0.7ex{
    \THH(A \times B) \ar@{.>}[dddd]_{(\text{pr}_A,\text{pr}_B)}& \ar[l]_(0.6)\simeq^(0.6){1)}
    \mathrm{hocolim}_{\Delta^{op}} \THH_\bullet(A \times B) & \ar[l]_\simeq^{2)}
    \mathrm{hocolim}_{\Delta^{op}_f} {\THH_\bullet(A \times B)}
    \\
 & & \mathrm{hocolim}_{\Delta^{op}_f} {\THH_\bullet(A \vee B)}  \ar[u]_\simeq^{3)} \ar@<1ex>[d]^R \\
 && \mathrm{hocolim}_{\Delta^{op}_f} {(\THH_\bullet(A) \vee \THH_\bullet(B))}
 \ar[d]_\simeq^{4)} \ar@<1ex>[u]^J\\
 & &  \mathrm{hocolim}_{\Delta^{op}_f} {\THH_\bullet(A)} \vee
 \mathrm{hocolim}_{\Delta^{op}_f} {\THH_\bullet(B)}
 \ar[d]_\simeq^{2)}\\
 \THH(A) \times \THH(B) & \ar[l]_\simeq^{5)} {\THH(A) \vee \THH(B)} &
 \mathrm{hocolim}_{\Delta^{op}} {\THH_\bullet(A)} \vee
 \mathrm{hocolim}_{\Delta^{op}} {\THH_\bullet(B)} \ar[l]_(0.65)\simeq^(0.65){1)}\\
    }
\]

1) With \cite[Proposition 4.27]{shipley} we obtain a level equivalence between 
$\THH(A\times B)$ and 
$\mathrm{hocolim}_{\Delta^{op}}\THH_\bullet(A \times B)$. Similarly, in the
bottom row we can identify the homotopy colimit with $\THH$. This does not need
any cofibrancy assumptions.

2) Let $\Delta_f$
denote the subcategory of $\Delta$ containing all objects but only injective
maps. The induced map on homotopy colimits 
\[ \mathrm{hocolim}_{\Delta^{op}_f} \THH_\bullet(A \times B) \ra
    \mathrm{hocolim}_{\Delta^{op}} \THH_{\bullet}(A \times B) \]
is an equivalence because the homotopy colimit of symmetric spectra is 
defined levelwise \cite[Definition 2.2.1]{shipley} and in every 
simplicial degree $p$ and every level $\ell$, $\THH_p(A \times B)(\ell)$ 
is cofibrant (because it is just a simplicial set) and hence the 
claim follows as in \cite[Proposition 20.5]{dugger}. An analoguous argument
applies in the second occurence of 2). 

3) 
We consider $A \vee B$ as a non-unital commutative ring spectrum via
the multiplication map
\[ (A \vee B) \wedge (A \vee B) \simeq (A \wedge A) \vee (A \wedge B) \vee
  (B \wedge A) \vee (B \wedge B) \ra (A \wedge A) \vee (B \wedge B)
  \ra A \vee B\]
where the first map sends the mixed terms to the terminal spectrum and the
second one uses the multiplication in $A$ and $B$. Correspondingly,
$\THH_\bullet(A \vee B)$ is a presimplicial spectrum, that only uses the face
maps of the structure maps of $\THH$.

The canonical map $A \vee B \ra A \times B$ is a stable equivalence of
non-unital symmetric ring spectra and hence by adapting the argument in the
proof of \cite[Corollary 4.2.9]{shipley} we get a $\pi_*$-equivalence of the
corresponding presimplicial objects
\[ \THH_\bullet(A \vee B) \simeq \THH_\bullet(A \times B).\]
Pointwise level equivalences give level equivalences on homotopy colimits
\cite[Proposition 2.2.2]{shipley}, so the map in 3) is a level equivalence.

4) Homotopy colimits commute with sums. 

5) The product is stably equivalent to the sum.

\smallskip
To have a chain of stable equivalences between  $\THH(A \times B)$ and $\THH(A) \times \THH(B)$ it remains to
understand the effect of the maps $J$ and $R$ and we will control them in Lemma \ref{actualhomotopy}
below.  

We claim that the product of the projections 
\[\text{pr}_A \colon  A\times B\to A, \quad \text{pr}_B\colon  A\times B \to B\] 
produces an equivalence.  Observe that we can apply the projection $\text{pr}_A$ to every stage in our diagram.
On $A\vee B$, this will induce the collapse map $A\vee B \to A$.
By applying $\text{pr}_A$ to the entire chain of equivalences,
we get a diagram of equivalences between various versions of $\THH(A)$.
We can do the same for $\text{pr}_B$.
This gives a commutative diagram
\[ \xymatrix{
\THH(A \times B)   \ar[d]_{(\text{pr}_A, \text{pr}_B)} &
\mathrm {chain\ of\  maps} 
&  \THH(A) \times \THH(B)  \ar[d]^{\text{pr}_A \times \text{pr}_B} \\
 \THH(A) \times \THH(B) &
 (\text{pr}_A( \mathrm{the\ chain}), \text{pr}_B(\mathrm{ the\ chain}))  
& \THH(A) \times \THH(B)  \\ 
}\]
where the chain is a zigzag of arrows going both ways.  The chain on top has all its stages equivalences.
By the above discussion, so is the product of the chains on the bottom.  And the map on the right is the identity.
So working step by step in the zigzag from the right, we show that the pair of projections
$(\text{pr}_A,\text{pr}_B)$ induces equivalences
at all the intermediate steps, until we get to the leftmost $(\text{pr}_A, \text{pr}_B)$ which is therefore also
an equivalence.
\end{proof}

We consider the map
\[j \colon \THH_\bullet(A) \vee \THH_\bullet(B) \ra
  \THH_\bullet(A \vee B)\]
that is induced by the inclusions $A \hookrightarrow A
\vee  B$ and $B \hookrightarrow A \vee B$. We let $J$ be the induced map
on the homotopy colimit. It has a retraction
$R = \mathrm{hocolim}_{\Delta^{op}_f}r$ with
\[ r
  \colon \THH_\bullet(A \vee B) \ra \THH_\bullet(A) \vee \THH_\bullet(B)\]
that sends all mixed
smash products to the terminal object. Note that $R \circ J = \mathrm{id}$.

\begin{lem}\label{actualhomotopy}
There is a presimplicial homotopy $j \circ r \simeq \mathrm{id}$.
\end{lem}
\begin{proof} 
We consider the $n$th presimplicial degree of $\THH_n(A \vee B)$:
\[ \THH_n(A \vee B) = (A \vee B)^{n+1}.\]
This is a sum of terms of the form $A^{\wedge i_1} \wedge B^{\wedge i_2} \wedge
\ldots \wedge B^{\wedge i_k} \wedge A^{\wedge i_{k+1}}$ for suitable $k$ with
  $0 \leq i_1, i_{k+1} $ and $0<i_j $ for $1<j<k+1$, so that $\sum_{j=.1}^{k+1} i_j=n+1$.

  Restricted to such a summand we define $h_j \colon   (A \vee B)^{n+1} \ra
  (A \vee B)^{n+2}$ for $0 \leq j \leq n$ as
\[ h_j|_{A^{\wedge i_1} \wedge B^{\wedge i_2} \wedge
    \ldots \wedge B^{\wedge i_k} \wedge A^{\wedge i_{k+1}}} =
    \begin{cases}
      \mathrm{id}_{A^{\wedge {j+1}}} \wedge \eta_A \wedge \mathrm{id}, &
      \text{ if } {j+1} \leq i_1, \\
      \mathrm{id}_{B^{\wedge {j+1}}} \wedge \eta_B \wedge \mathrm{id}, & \text{ if }
      i_1=0 \text{ and } {j+1} \leq i_2, \\
      *, & \text{ otherwise.}
    \end{cases}
\]
Then $d_0 h_0 = \mathrm{id}$, $d_{n+1}h_n = j\circ r$ and 
\[ d_ih_j = 
\begin{cases}
h_{j-1}d_i, & i < j, \\
d_ih_{j-1}, & i = j\neq 0, \\
h_jd_{i-1}, & i > j+1. 
\end{cases}
\]
\end{proof}  
This implies that $J \circ R \simeq \mathrm{id}$, so we get that 
\[ \mathrm{hocolim}_{\Delta^{op}_f}\THH_\bullet(A \vee B) \simeq
 \mathrm{hocolim}_{\Delta^{op}_f} (\THH_\bullet(A) \vee \THH_\bullet(B)).\]

For the general setting we work with \emph{commutative} ring spectra and we
return to the setting of \cite{ekmm}. 
Note that the naturality
of cofibrant replacements ensures that we get morphisms of commutative
ring spectra 
\[ A^c \leftarrow (A \times B)^c \ra B^c\]
and hence a weak equivalence (because the product of acyclic fibrations is an acyclic fibration). 
\[ (A \times B)^c \ra A^c \times B^c.\]
This proves the case of $X = *$ of the following proposition and is needed in the proof. 
\begin{prop} \label{prop:product} 
For any connected finite simplicial set $X$, the projection  maps
$\text{pr}_A \colon A\times B$ and $\text{pr}_B 
\colon A\times B \ra B$ induce an equivalence
\[ \L_X((A\times B)^c) \simeq \L_X(A^c)\times \L_X(B^c)\]
 and in particular, for all $n\geq 1$,
\[\THH^{[n]}((A\times B)^c) \simeq \THH^{[n]}(A^c) \times \THH^{[n]}(B^c).\]
\end{prop}

\begin{proof}
We prove the result for all finite connected simplicial sets $X$
by induction on the dimension $n$ of the top non-degenerate simplex in
$X$. Since the only finite connected simplicial set with its only
non-degenerate simplices in dimension zero is a point, the result is
obvious for $n=0$. 

\smallskip
For higher $n$, the crucial observation is that if we have simplicial
sets $X$, $Y$, and $Z$ so that $Z$ is a non-empty subset of both $X$
and $Y$, then if the projection maps $(A \times B)^c \ra A^c$ and
$(A \times B)^c \ra B^c$ 
induce equivalences as given in the statement of this proposition for 
$X$, $Y$ and $Z$, then we also obtain an equivalence 
\begin{equation}\label{eqn:glue}
 \xymatrix@1{
{\L_{X\cup_Z Y}((A \times B)^c)} 
\ar[r]^(0.4){\simeq} &
 {\L_{X\cup_Z Y}(A^c) \times \L_{X\cup_Z Y}(B^c). } }
\end{equation}
This is because then 
\begin{align*}
\L_{X\cup_Z Y}((A \times B)^c) & \simeq  \L_X((A\times B)^c) \wedge_{\L_Z((A\times B)^c)}
    \L_Y((A\times B)^c) \\
 & \simeq (\L_X(A^c)\times \L_X(B^c)) \wedge^L_{\L_Z(A^c)\times \L_Z(B^c)}
   (\L_Y(A^c)\times \L_Y(B^c))\\ 
& \simeq 
(\L_X(A^c)\wedge_{\L_Z(A^c)} \L_Y(A^c)) \times (\L_X(B^c)\wedge_{\L_Z(B^c)} \L_Y(B^c))\\
& \simeq  \L_{X\cup_Z Y} (A^c)\times \L_{X\cup_Z Y}(B^c). 
\end{align*}
For the first and last equivalence we use that the Loday construction
sends pushouts to homotopy pushouts and that for a cofibrant commutative ring spectrum
the map $\L_Z(R) \ra \L_Y(R)$ is a cofibration. For the second equivalence we use our assumption that
the proposition holds for $X$, $Y$ and $Z$.

The third equivalence 
holds because $\L_Z(A^c)$ acts trivially on $\L_X(B^c)$ and on $\L_Y(B^c)$ thus
it sends the corresponding factors to the terminal ring spectrum; the same
holds for the action of $\L_Z(B^c)$ on $\L_X(A^c)$ and on $\L_Y(A^c)$. Therefore a K\"unneth
spectral sequence argument shows that we obtain a weak equivalence.

\smallskip
For $n=1$, we use homotopy invariance of the Loday construction and the
fact that any  finite connected simplicial set with non-degenerate
simplices only in dimensions $0$ and $1$ is homotopy equivalent  
to $\bigvee_{i=1}^m S^1$ for some $m\geq 0$.  If $m=0$, we deduce the
proposition 
from the $n=0$ case above; if $m=1$, we use Proposition \ref{prop:THHsum},
and for $m>1$, we use induction and Equation (\ref{eqn:glue}).  

\smallskip
For the inductive step, assume that $n>1$ and we know that the
proposition holds for any finite connected simplicial set with
non-degenerate cells in dimensions $<n$, and in particular for
$\partial \Delta^n$, the boundary of the standard $n$-simplex.  Assume
that we have a finite simplicial set $X$ for which the proposition holds.  We
then prove that the proposition also holds for $X \cup_{\partial
  \Delta^n} \Delta^n$, that is: $X$ with an additional $n$-simplex
glued to it along the boundary.  Without loss of generality, we may
assume that the boundary of the new $n$-simplex is embedded in $X$: if
it is not, apply four-fold edgewise subdivision to everything, and
then $X \cup_{\partial \Delta^n} \Delta^n$ will consist of the central
small $n$-simplex inside the original $n$-simplex that was added that
does not touch the boundary of the originally added $n$-simplex and all
the rest of the subdivided complex.  But the rest of the subdivided
complex is homotopy equivalent to the original $X$, so the proposition
holds for it, and the central small $n$-simplex does indeed have
its boundary embedded in the four-fold edgewise subdivision of  
$X \cup_{\partial \Delta^n} \Delta^n$.

Then by assumption, the proposition holds for $X$, by the inductive
hypothesis it holds for $\partial \Delta^n$, by homotopy invariance it
holds for $\Delta^n\simeq *$, and so by Equation (\ref{eqn:glue}) it
holds for $X \cup_{\partial \Delta^n} \Delta^n$. 
\end{proof}
For later use we need a version of Proposition \ref{prop:product} with 
coefficients. Again, we choose cofibrant models $(A \times B)^c$, $A^c$ and $B^c$
and we assume $M^c$ is a cofibrant $A^c$-module spectrum, $N^c$ is a cofibrant $B^c$-module
spectrum and $(M \times N)^c$ is a cofibrant $(A \times B)^c$-module spectrum, such that
these cofibrant replacements are compatible with the projection maps on $(A \times B)^c$ and
$(M \times N)^c$. 
\begin{cor} \label{cor:sums}
For all connected pointed finite simplicial sets there is an
equivalence 
\[\L_X((A \times B)^c; (M \times N)^c) \ra \L_X(A^c; M^c) \times \L_X(B^c; N^c).\]
\end{cor}
\begin{proof}
The argument in the proof of Proposition  \ref{prop:product} can be
adapted to pointed finite simplicial sets. We know that 
\[ \L_X((A \times B)^c; (M \times N)^c) \simeq \L_X((A \times B)^c)
  \wedge_{(A \times B)^c} ((M \times N)^c)\]
and by the result above this is equivalent to 
\[ (\L_X(A^c) \times \L_X(B^c)) \wedge^L_{A^c \times B^c} (M^c \times N^c)\]
Again, we can identify the coequalizers because the action of $A^c$ on $N^c$ and the
one of $B^c$ on $M^c$ is trivial and obtain
\[ (\L_X(A^c) \wedge_{A^c} M^c) \times (\L_X(B^c) \wedge_{B^c} N^c) \simeq
  \L_X(A^c;M^c) \times \L_X(B^c; N^c).\]
  
\end{proof}

\section{Applications} \label{sec:appl}

\subsection{$\THH^{[n]}(\Z/p^m;\Z/p)$} \label{subsec:thhnzmpzp}
This example was our original motivation for obtaining the splitting
result of Theorem~\ref{thm:Zp^m}. We apply it to the case where $R = \Z$,
$p$ is a prime, and $a= p^m$ for $m\geq 2$. As a special case
of Theorem \ref{thm:Zp^m} we obtain the following splitting. 
\begin{thm} \label{thm:thhnzpmzp}
\begin{align*}
\THH^{[n]}(\Z/p^m;\Z/p) &\simeq \THH^{[n]}(\Z;\Z/p) \wedge^L_{H\Z/p}
  \THH^{[n],\Z}(\Z/p^m;\Z/p) \\&\cong \THH^{[n]}(\Z;\Z/p) \wedge^L_{H\Z/p}
  \Sh^{[n],\Z}(\Z/p^m;\Z/p).
\end{align*}
\end{thm}
This gives a direct calculation of $\THH_*^{[n]}(\Z/p^m;\Z/p)$ for all
$n$ because in \cite[Theorem 3.1]{dlr} we determine 
$\THH_*^{[n]}(\Z;\Z/p)$ as an iterated Tor-algebra $B^n_{\F_p}(x)
\otimes_{\F_p} B^{n+1}_{\F_p}(y)$ where $|x|= 2p$, $|y|=2p-2$,
$B^1_{\F_p}(z) = \F_p[z]$ and $B^n_{\F_p}(z) =
\Tor_*^{B^{n-1}_{\F_p}(z)}(\F_p, \F_p)$. We determined
$\Sh^{[n],\Z}(\Z/p^m;\Z/p)$ in Proposition \ref{prop:shuklazzp}. 

The case $n = 1$ was calculated in \cite{pirashvili}, \cite{brun} and
later as well in \cite{angeltveit}. 

\begin{rem}
Note that we cannot use the sequence 
canonical projection maps 
$$ \xymatrix@1{{\ldots} \ar[r] & {\Z/p^{m+1}\Z} \ar[r] & {\Z/p^m\Z}
  \ar[r] & {\ldots} \ar[r] & {\Z/p^2\Z} \ar[r] & {\Z/p\Z}}$$
in order to compare the groups $\THH(H\Z/p^m\Z; H\Z/p\Z)$ for varying
$m$ because the tensor factors coming from $\THH_*(H\Z; H\Z/p\Z)$ are mapped
isomorphically from $\THH(H\Z/p^{m+1}\Z; H\Z/p\Z)$ to $\THH(H\Z/p^m\Z;
H\Z/p\Z)$ whereas the tensor factor $\Sh^{[n],\Z}(\Z/p^{m+1};\Z/p)$ is
mapped via the augmentation map to $\Sh^{[n],\Z}(\Z/p^m;\Z/p)$ in each
step of the sequence.   
This is straightforward to see with the help of the explicit 
resolutions used in the proof of Lemma  \ref{lem:ap->pp}. 
\end{rem}

\subsection{Number rings} \label{subsec:numberrings}
As a warm-up we consider $R = \Z[i]$, $p = 1-i$ and $2 \in (p)^2$.  Then
we get that 
\[\THH^{[n]}(\Z[i]/2; \Z[i]/(1-i)) \simeq \THH^{[n]}(\Z[i];\Z[i]/(1-i))
  \wedge_{H\Z[i]/(1-i)} \Sh^{[n],\Z[i]}(\Z[i]/2;\Z[i]/(1-i)).\]
Note that $\Z[i]/(1-i) \cong \Z/2$ and $\Z[i]/2 \cong \F_2[x]/x^2$.
Thus we can calculate  
$$\THH^{[n]}(\Z[i]/2;\Z[i](1-i)) \cong \THH^{[n]}(\F_2[x]/x^2;\F_2)$$
using the flow chart in \cite{blprz} and we know from \cite[Theorem
4.3]{dlr} that $\THH_*^{[n]}(\Z[i];\Z[i]/(1-i))$  can also be computed using
iterated $\Tor$'s. 
The term
$\Sh^{[n],\Z[i]}(\Z[i]/2;\Z[i]/(1-i))$ can be computed as an iterated
$\Tor$ by Proposition \ref{prop:shuklazzp}. 
Thus all of the components of the above expression are known. What was
not known before is that $\THH^{[n]}(\Z[i]/2; \Z[i]/(1-i))$ splits in
the above manner. 

The general case is as follows: 
Consider $p\in \Z$ a prime, and let $K$ be a number field such that
$p$ is ramified 
in $\O_K$, with $p = \p_1^{e_1}\cdots \p_r^{e_r}$ with one $e_i >
1$. The Chinese Remainder Theorem let's us split $\O_K/p$ as a ring as 
$$ \O_K/p \cong \prod_{j=1}^r \O_K/\p_j^{e_j}$$ and $\O_K/\p_i$ as an
$\O_K/p$-module is then isomorphic to $0\times\ldots \times \O_k/\p_i
\times \ldots\times 0$ with the non-trivial component sitting in spot
number $i$. 
With Corollary \ref{cor:sums} we obtain the following
result. 

\begin{thm} 
\begin{equation} \label{eq:THHOK}
  \THH^{[n]}(\O_K/p;\O_K/\p_i) \simeq \THH^{[n]}(\O_K;\O/\p_i)
  \wedge_{H\O_K/\p_i} \THH^{[n],\O_K}(\O_K/p;\O_K/\p_i).
\end{equation}
\end{thm}
Again, $\O_K/p \cong (\O_K)_{\p_i}/p$ is isomorphic to
$\O_K/\p_i[\pi]/\pi^{e_i}$ where $\pi$ is the uniformizer, hence
$\O_K/\p_i[\pi]/\pi^{e_i} \cong \O_K/\p_i[x]/x^{e_i}$ so we can
calculate $\THH^{[n]}(\O_K/p;\O_K/\p_i)$ using Proposition
\ref{prop:reduced}. We can determine $\THH^{[n]}(\O_K;\O/\p_i)$ using
\cite[Theorem 4.3]{dlr} and we calculated
$\THH^{[n],\O_K}(\O_K/p;\O_K/\p_i)$ in Proposition 
\ref{prop:shuklazzp}. Using these calculations one can deduce right away
that there is a splitting on the level of homotopy groups of
$\THH^{[n]}_*(\O_K/p;\O_K/\p_i)$. But \eqref{eq:THHOK} yields a splitting
of  $\THH^{[n]}(\O_K/p;\O_K/\p_i)$ on the  level of augmented
commutative $H\O_K/\p_i$-algebra spectra.

\subsection{Galois descent} \label{subsec:gd}
In \cite[Definition 9.2.1]{rognes} John Rognes defines a map of
commutative $S$-algebras $f \colon A \ra B$ to be formally
$\THH$-\'etale, if the unit map  
$B \ra \THH^A(B)$ is a weak equivalence. Note that this  implies
that the augmentation map $\THH^A(B) \ra B$ that is induced by
multiplying all $B$-entries in $\THH^A(B)$ together, is also a weak
equivalence because the compositve $B \ra \THH^A(B) \ra B$ is the
identity on $B$. 

Therefore, applying the Brun juggling formula \ref{thm:brun} in this
case to $X = S^1$ we obtain
$$ \THH^{[2]}(A) \wedge^L_A B 
\simeq \THH^{[2]}(A; B) \simeq B \wedge_{\THH(B)} \THH^A(B) 
\simeq   B \wedge^L_{\THH(B)} B 
\simeq  \THH^{[2]}(B). $$

We can slightly generalize this: 
\begin{defn}
Let $X$ be a pointed simplicial set. A morphism $f \colon A \ra B$ is
formally  $X$-\'etale, if 
the unit map $B \ra \L_X^A(B)$ is a weak equivalence. 
\end{defn}

For formally $X$-\'etale morphisms $f \colon A \ra B$ the Brun
juggling formula \ref{thm:brun} for $X$ implies  
$$ \L_{\Sigma X}(A) \wedge^L_A B \simeq \L_{\Sigma X}(A; B) \simeq B
\wedge_{\L_X(B)} \L_X^A(B) \simeq   B \wedge_{\L_X(B)} B \simeq \L_{\Sigma X}(B). $$

This statement is related to Akhil Mathew's result \cite[Proposition
5.2]{mathew} where he shows that $\L_Y(A; B) \simeq \L_Y(B)$ if $f
\colon A \ra B$ is a faithful finite $G$-Galois extension and if $Y$
is a simply-connected pointed simplicial set. Such Galois
extensions are formally  $\THH$-\'etale by \cite[Lemma 9.2.6]{rognes}.

\subsection{Algebraic function fields over $\F_p$} \label{subsec:ff}
In several of our splitting formulas higher $\THH$ of the ground field
is a tensor factor. So far we have only considered prime fields or
rather simple-minded algebraic extensions of those. Topological
Hochschild homology groups of algebraic function
fields are an important class of examples. 

Let $L$ be an algebraic function field over $\F_p$. Then there is a 
transcendence basis $(x_1,\ldots, x_d)$ such that $L$ is a finite separable 
extension of $\F_p(x_1, \ldots, x_d)$ \cite[Theorem 9.27]{milne}. As 
separable extensions do not contribute anything substantial to topological 
Hochschild homology we obtain the following result:
\begin{thm} \label{thm:thhff}
Let $L$ be an algebraic function field over $\F_p$, then 
$$ \THH(L)_*  \cong 
L \otimes_{\F_p} \THH_*(\F_p) \otimes_{\F_p}  
\Lambda_{\F_p}(\varepsilon x_1, \ldots, \varepsilon x_d).
$$ 
\end{thm}
\begin{proof}
McCarthy and Minasian show in \cite[5.5, 5.6]{mccm} that $\THH$ has
\'etale descent in our case. Therefore 
\begin{align*}
\THH(L) & \simeq HL \wedge^L_{H\F_p(x_1, \ldots, x_d)} \THH(\F_p(x_1,
\ldots, x_d)) \\
&  \simeq HL \wedge^L_{H\F_p(x_1, \ldots, x_d)} H\F_p(x_1,
\ldots, x_d) \wedge^L_{H\F_p[x_1, \ldots, x_d]} \THH(\F_p[x_1, \ldots,
x_d]) \\
& \simeq HL \wedge^L_{H\F_p[x_1, \ldots, x_d]} \THH(\F_p[x_1, \ldots,
x_d]).
\end{align*}
But the topological Hochschild homology of monoid rings is known by
\cite[Theorem 7.1]{hm} and hence $\pi_*\THH(\F_p[x_1, \ldots,
x_d]) \cong \THH_*(\F_p) \otimes_{\F_p} \HH_*(\F_p[x_1, \ldots,
x_d])$. As 
$$\HH_*(\F_p[x_1, \ldots, x_d]) \cong \F_p[x_1, \ldots,
x_d] \otimes_{\F_p}  
\Lambda_{\F_p}(\varepsilon x_1, \ldots, \varepsilon x_d)$$
we get the result. 
\end{proof}
McCarthy and Minasian actually show more in \cite[5.5,5.6]{mccm}, and
we can adapt the above  
proof to a more general situation. 
\begin{thm} 
Let $X$ be a connected pointed simplicial set $X$. Then 
$$\L_X(HL)
\simeq HL \wedge^L_{H\F_p[x_1, \ldots, x_d]} \L_X(\F_p[x_1, \ldots,
x_d]).$$ 
\end{thm}
The Loday construction on pointed monoid algebras satisfies a
splitting of the form 
$$ \L_X(H\F_p[\Pi_+]) \simeq \L_X(H\F_p) \wedge_{H\F_p}
\L_X^{H\F_p}(H\F_p[\Pi_+]), $$
see \cite[Theorem 7.1]{hm}. Therefore $\L_X(\F_p[x_1, \ldots,
x_d])$ splits as $\L_X(H\F_p) \wedge_{H\F_p}
\L_X^{H\F_p}(H\F_p[x_1,\ldots, x_d])$. 

In particular, for $X = S^n$ we get an explicit formula for
$\THH^{[n]}(L)$: 
\begin{cor}
For all $n \geq 1$: 
$$ \THH^{[n]}(L) \simeq HL \wedge^L_{H\F_p[x_1, \ldots, x_d]}
(\THH^{[n]}(\F_p) \wedge_{H\F_p} \THH^{[n], \F_p}(\F_p[x_1, \ldots,
x_d])) .$$
\end{cor}
Recall that we know
\begin{align*}
\pi_*\THH^{[n], \F_p}(\F_p[x_1, \ldots, x_d]) & = \HH^{[n], \F_p}_*(\F_p[x_1, \ldots,
x_d]) \\ 
 & \cong  \HH^{[n], \F_p}_*(\F_p[x]^{\otimes_{\F_p} d}) \\
 & \cong \HH^{[n], \F_p}_*(\F_p[x])^{\otimes_{\F_p} d}
\end{align*}
and we determined $\HH^{[n], \F_p}_*(\F_p[x])$ in 
\cite[Theorem 8.6]{blprz}. 

\begin{rem}
Topological Hochschild homology of $L$ considers $HL$ as an $S$-algebra 
and this allows us to consider $L$ over the prime field. 
The Hochschild homology of an algebraic function field $L$ over a 
general field $K$ was for instance determined in \cite[Corollary 5.3]{BK} and 
is more complicated. 
\end{rem}
\begin{rem}
Note that Theorem \ref{thm:thhff} contradicts the statement of \cite[Remark
7.2]{greenlees}. In that remark, it is crucial to assume that one works
in an \emph{augmented} setting; in the above situation, this is not the case. 
\end{rem}

\subsection{$\THH^{[n]}(\Z/p^m)$} \label{subsec:thhnzpm}
We close with the open problem of computing $\THH^{[n]}(\Z/p^m)$ for
higher $n$. 

The juggling formula \ref{thm:juggling} applied to the sequence $S \ra
H\Z \ra H\Z/p^m = H\Z/p^m$ for $m \geq 2$ yields the equivalence 
$$ \THH^{[n]}(\Z/p^m) \simeq \THH^{[n]}(\Z; \Z/p^m)
\wedge^L_{\THH^{[n-1],\Z}(H\Z/p^m)} H\Z/p^m.  $$ 
Up to $n = 2$ we know $\THH_*^{[n]}(H\Z)$: The case $n=1$ is
B\"okstedt's calculation \cite{boekstedt} and $n=2$ is \cite[Theorem
2.1]{dlr2}. Therefore we can determine $\THH_*^{[n]}(\Z; \Z/p^m)$ up
to $n=2$. As
$p^m$ is regular in $\Z$,  
$$\THH_*^\Z(\Z/p^m) = \Sh_*^\Z(\Z/p^m) \cong \Gamma_{\Z/p^m}(x_2)$$
with $|x_2|= 2$. If we could determine the right
$\Sh_*^\Z(\Z/p^m)$-module structure on $\THH_*^{[2]}(\Z; \Z/p^m)$,
then this would allow us to calculate the $E^2$-term of the K\"unneth
spectral sequence for $\THH_*^{[2]}(\Z/p^m)$, 
\[ E^2_{p,q} = \text{Tor}_{p,q}^{\Gamma_{\Z/p^m}(x_2)}(\THH_*^{[2]}(\Z; \Z/p^m),
\Z/p^m) \Rightarrow \THH_*^{[2]}(\Z/p^m). \]

\end{document}